\documentclass[a4paper,10pt]{article}

\usepackage{a4wide}
\usepackage{amsmath}
\usepackage{amsthm}
\usepackage{amsfonts}
\usepackage{amssymb}
\usepackage{mathtools}
\usepackage{algorithm}
\usepackage{epsfig}
\usepackage[utf8]{inputenc}
\usepackage{graphics}
\usepackage{caption}
\usepackage{booktabs}
\usepackage{framed}
\usepackage{listings}
\usepackage{caption}
\usepackage{subcaption}
\usepackage{colortbl}
\usepackage{xcolor}
\usepackage[colorlinks]{hyperref}
\usepackage{color} 
\usepackage[most]{tcolorbox}
\usepackage{algpseudocode}

\newtcolorbox{tcbdoublebox}[1][]{%
  enhanced jigsaw,
  sharp corners,
  colback=white,
  borderline={1pt}{-2pt}{black},
  fontupper={\setlength{\parindent}{20pt}},
  #1
}

\newcolumntype{g}{>{\columncolor{Gray}}c}
\definecolor{mygreen}{RGB}{28,172,0} 
\definecolor{mylilas}{RGB}{170,55,241}

\usepackage{fancyhdr} 
\newcommand{\R}{\mathbb{R}}

\newcommand{\E}{\mathbb{E}}

\newcommand{\ra}{\rightarrow}

\newcommand{\bmat}{\begin{bmatrix}}
\newcommand{\emat}{\end{bmatrix}}
\newcommand{\ljp}{[\![}
\newcommand{\rjp}{]\!]}

\newcommand{\cD}{{\mathcal{D}}}
\newcommand{\cW}{{\mathcal{W}}}
\newcommand{\cQ}{{\mathcal{Q}}}

\newcommand{\cT}{{\mathcal{T}}}

\newcommand{\Em}{{\mathbb{E}_{\xi}}}

\DeclareMathOperator*{\esssup}{ess\,sup}

\DeclareMathOperator{\Div}{div}
\DeclareMathOperator{\Tr}{tr}
\DeclareMathOperator{\diff}{d\!}

\newtheorem{theorem}{Theorem}

\newtheorem{assumption}{Assumption}

\newtheorem{definition}{Definition}
\newtheorem{remark}{Remark}
\newtheorem{proposition}{Proposition}
\newtheorem{acknowledgements}{Acknowledgements}

\newcommand{\x}[2][]{#2}

\definecolor{Gray}{gray}{0.9}

\author{Oğuz Han Altıntaş \footnote{Institute of Applied Mathematics, Middle East Technical University, 06800, Ankara, Türkiye and  ROKETSAN, 06780, Ankara, Türkiye, \url{oguz.altintas@metu.edu.tr}} \and Hamdullah Yücel \footnote{Institute of Applied Mathematics, METU, Ankara, Türkiye, \url{yucelh@metu.edu.tr}}}
\title{An Adaptive Framework for Robust Structural Shape Optimization under Uncertainty}

\begin{document}

\maketitle

\begin{abstract}
This work presents an adaptive framework for solving a robust structural shape optimization problem governed by linear elasticity with uncertain loading and material parameters. A posteriori error estimators are constructed to control the sample size, mesh resolution, and optimization step length. The sample size used in the stochastic gradient approximation is adjusted dynamically according to the variance of the sampled shape derivatives. In the physical domain, the proposed error estimation strategy accounts not only for discretization errors in the elasticity constraint but also for errors arising from the discretization of the deformation problem used to compute descent directions. The optimization step length is determined adaptively through an estimate of the Lipschitz constant of the stochastic shape derivative. Moreover, existence results and a distributed representation of the stochastic shape derivative are established. Finally, the proposed adaptive stochastic optimization framework is validated on leg-like structural components, demonstrating its effectiveness in minimizing touchdown compliance under uncertain contact forces.
\end{abstract}





\section{Introduction} \label{sec:introduction} 

Structural shape optimization is a powerful technique for enhancing complex scientific and engineering designs by optimizing material layout through variations in boundaries and connectivity within the design domain; see, e.g., \cite{OPironneau_1984,JSokolowski_JPZolesio_1992a} for a comprehensive introduction. In many studies within the literature, input data such as loading conditions or material properties are typically assumed to be deterministic. However, due to factors such as manufacturing processes, unknown loading conditions, variations in material properties, and lack of information, these problems inherently contain various types of uncertainties. These uncertainties may be related to loading, materials, geometry, or model-based factors \cite{RCSmith_2013}. Therefore, it is crucial to account for these uncertainties to ensure a robust and reliable design.

Over the last two decades, researchers have delved into shape and/or topology optimization problems that incorporate uncertainty. In works such as \cite{GAllaire_CDapogny_2014,FDeGournay_GAllaire_FJouve_2008}, methods such as the worst-case approach and fuzzy techniques do not rely on statistical information about the underlying uncertainty; instead, they use a qualitative measure of its magnitude. Findings from these studies indicate that the resulting optimal designs often lead to poor structural performance. Methods that use statistical information represented by the statistical moments of the structural response (e.g., mean and variance) or the tails of probability distributions yield promising results. The former  is known as robust design optimization \cite{HGBeyer_BSendhoff_2007a,BLazarov_MSchevenels_OSigmund_2012a,MTootkaboni_AAsadpoure_JKGuest_2012}, which focuses on creating optimal designs that are less sensitive to variations in the input data. This type of measure is useful if the focus is on achieving good average performance while placing less emphasis on variability.  The latter approach is referred to as reliability-based design optimization
\cite{SConti_HHeld_MPach_MRumf_RSchultz_2011,DWKim_BMKwak_1996}. This type of risk-averse formulation focuses on managing extreme or rare events to achieve a specified level of reliability using probabilistic constraints. However, this is beyond the scope of the present study. Instead, we mainly investigate a numerical approximation of the robust structural shape optimization problem formulated with a (risk-neutral) expectation measure, accounting for  uncertainties in loading and material inputs within a linear elasticity model.

Compared to solving deterministic shape optimization problems, robust shape optimization problems pose more complex challenges due to the difficulty of discretizing  both physical and stochastic domains. Up to now, a great deal of effort has been devoted to the development of efficient methods for solving PDEs with uncertain inputs, based on projection methods (e.g., polynomial chaos \cite{DXiu_GEKarniadakis_2002} and stochastic Galerkin (SG)  \cite{IBabuska_RTempone_GEZouraris_2004a}),  and sampling methods (e.g., Monte Carlo \cite{GSFishman_1996}, stochastic collocation (SC) \cite{IBabuska_FNobile_RTempone_2007a}). In the projection approaches,  one approximates the PDE solution by projecting it onto a  basis that is inexpensive to compute, whereas sample-based approaches generate a finite set of realizations of the PDE solution that are then used to compute statistics or interpolated to build an approximation of the full solution. In the realm of  robust structural shape and topology optimization, we refer to \cite{BLazarov_MSchevenels_OSigmund_2012a} for the stochastic collocation method, \cite{MTootkaboni_AAsadpoure_JKGuest_2012} for the stochastic Galerkin approach, \cite{GAllaire_CDapogny_2015} for the perturbation methods, and \cite{SDe_JHampton_KMaute_ADoostan_2020} for the Monte Carlo technique. Other studies on robust shape optimization include  \cite{SConti_HHeld_MPach_MRumf_RSchultz_2009}, which uses a two-stage stochastic programming algorithm; \cite{JMFrutos_DHPerez_MKessler_FPeriago_2016}, which utilizes an approach based on a dimension-adaptive sparse grid, and \cite{CAudouze_AKlein_AButscher_NMorris_PNair_MYano_2023}, which is based on a non-intrusive anchored ANOVA Petrov-Galerkin projection scheme. Although the polynomial chaos or SC methods, generally,  exhibit better convergence behavior for low-dimensional uncertainties,  a larger number of stochastic dimensions due to the low regularity of random fields may cause a prohibitively large algebraic system, known as the curse of dimensionality.  Therefore, we concentrate on Monte Carlo sampling methods (more generally, sample average approximation), which allow us to effectively represent the stochastic problem as a finite summation, even when dealing with a large number of uncertainties. 

Optimization problems with uncertainty governed by PDEs have been studied using deterministic optimization methods
in combination with sampling or discretization schemes in the stochastic space. However, using a gradient descent approach requires solving $2N$ PDEs (i.e., $N$ state and $N$ adjoint state equations) at each iteration, a cost that is rarely affordable. To reduce computational complexity, several contributions to stochastic optimization have been made, although they remain limited; see, e.g., \cite{CGeiersbach_WWollner_2020a,EHaber_MChung_FHerrmann_2012a,MMartin_SKrumscheid_FNobile_2021a,SCToraman_HYucel_2023} for PDE-constrained optimization  problems and \cite{SDe_KMaute_ADoostan_2020,LJofre_ADoostan_2022} for  shape optimization problems. For  \x{stochastic optimization} problems, it is well-known that the variance of the Monte Carlo estimator of  the quantity of interest (QoI) is inversely proportional to the number of samples $N$ and thus a large sample size leads to a high probability of obtaining an optimal solution for the QoI. However, using a large sample size at each optimization iteration is unnecessarily expensive in terms of computational cost. For this reason, it is useful to employ an adaptive sampling strategy as discussed in  \cite{RBollapragada_RHByrd_JNocedal_2018,RHByrd_GMChin_JNocedal_2012} to keep the sample size as small as possible.  In this methodology, a relatively small sample size is chosen at the beginning, and if the process is likely to produce an improvement in the QoI, the sample size remains unchanged. Otherwise, a larger sample size is selected according to an a posteriori error indicator related to the variance of the sampled QoI. One advantage of this process is that it provides rapid progress in the early stages through the use  of a small initial sample size, whereas a larger sample size  yields high accuracy in the solution as needed. For each sample, it is also essential to account for errors arising from the approximation of the domain and the numerical solution of the PDE constraint, as well as their derivatives within the finite element framework. To minimize computational complexity while preserving high accuracy in the geometric description of the optimal design and the approximation of the objective functional, it is crucial to keep the degrees of freedom (DoF) as low as possible and  to ensure that the shape derivative is strictly negative along the specified descent direction. To address these concerns, \emph{a posteriori} error estimators can provide valuable insights during the optimization process. Several works in the literature have highlighted the considerable potential of a posteriori error estimators for achieving optimal designs in shape and topology optimization challenges. The Zienkiewicz–Zhu error estimator was  employed in \cite{NBanichuk_FJBarthold_AFalk_EStein_1995} to improve the accuracy of solutions arising from the approximation of the underlying differential problem, while utilizing the norm of the design gradient of the Lagrange function as an a posteriori error estimator for optimality conditions. This research was later extended through the dual-weighted residual method to control the PDE error, while the Laplace–Beltrami error indicator was adopted to account for geometric errors associated with domain deformations in \cite{PMorin_RHNochetto_MSPauletti_MVerani_2012}.
Furthermore, we refer to  \cite{MBruggi_MVerani_2011,KNoboru_KYChung_TToshikazu_JETaylor_1986,ASchleupen_KMaute_ERamm_2020} for the application of adaptive mesh refinement in the shape and topology optimization problems. 

The principal contribution of this work lies in the development of an efficient adaptive framework for solving robust structural shape optimization problems governed by a linear elasticity model that incorporates uncertainties in both loading and material parameters. The proposed approach not only determines the sample size adaptively but also employs a posteriori error estimators based on the dual-weighted residual (DWR) method \cite{RBecker_RRannacher_2001a} to enhance the efficiency of the finite element analysis, thereby improving the numerical approximation beyond what can be achieved through fixed mesh refinement. In constructing the a posteriori error estimator, errors arising from the discretization of the deformation bilinear form, which provides a descent direction, are taken into account in addition to those resulting from the discretization of the constraint PDE, namely,  the linear elasticity system. Furthermore, during the mesh deformation process guided by the descent direction, the step length within the gradient-based optimization procedure is adaptively adjusted by estimating the Lipschitz constant, thereby promoting stable and efficient convergence.

The remainder of this paper is organized as follows. In Section~\ref{sec:model}, we introduce the linear elasticity problem formulated as a PDE constraint, together with the compliance objective functional subject to a volume fraction constraint, under suitable assumptions and notation. The shape differentiability of the expected compliance objective and the derivation of its shape derivative, subject to a penalized volume constraint, are presented in Section~\ref{sec:shape_diff}.  Section~\ref{sec:discretization} discusses the numerical techniques used to discretize the infinite-dimensional optimization problem, thereby yielding its finite-dimensional approximation, and describes the corresponding optimization procedures.  In addition, the a posteriori error estimators used to adaptively control the sample size, mesh refinement, and optimization step length are introduced.  Section~\ref{sec:algorithm} then presents the proposed adaptive algorithm, which incorporates the dual-weighted residual (DWR) method for multiple goal functionals and dynamically updates the sample size. To demonstrate the effectiveness and robustness of the proposed estimation-based adaptive stochastic optimization framework, Section~\ref{sec:numeric} presents a set of numerical experiments for a benchmark compliance minimization problem with randomly generated input data. The focus is on the shape optimization of leg-like structural components aimed at minimizing compliance at touchdown under uncertain contact-force conditions. The performance of the algorithm is examined in three settings: (i) a randomness study conducted on a fixed mesh, (ii) a mesh refinement study using full Monte Carlo sampling, and (iii) a fully adaptive framework that combines  adaptive sampling and adaptive mesh refinement. Finally, concluding remarks and potential directions for future research are discussed in Section~\ref{sec:conclusion}.

\section{Model problem} \label{sec:model} 

Throughout this paper, we adhere to the conventional notation for Sobolev spaces $W^{m,p}(\cD)$, with $m \in \mathbb{N}_0$ and $1 \leq p \leq \infty$, as outlined in standard references such as \cite{SCBrenner_LRScott_2008,PGCiarlet_1978a}, equipped with the norm $\| \cdot \|_{m,p,\cD}$ and the seminorm  $| \cdot |_{m,p,\cD}$ on an open and bounded polygonal  physical domain $\cD \subset \R^2$ with \x{Lipschitz} boundary $\partial \cD$. In particular, we denote $W^{m,2}(\cD)$ by $H^m(\cD)$ and the corresponding norm and seminorm  by $\| \cdot \|_{m,\cD}$ and $| \cdot |_{m,\cD}$, respectively. The spaces of square-integrable functions on $\cD$ and $\partial \cD$ are denoted by $L^2(\cD)$ and $L^2(\partial \cD)$, respectively, with norms $\| \cdot \|_{\cD}$ and $\| \cdot \|_{\partial \cD}$. \x{Throughout, $\|\cdot\|$ is the Euclidean norm on finite-dimensional Euclidean spaces, and for $x, y\in \R^n$ $x \cdot y = x^T y$ is the associated inner product.} $C^k(\R^2, \R^2)$  represents the space of $k$-times continuously differentiable vector-valued functions and $C>0$ is a generic constant independent of any discretization parameters.  A complete probability space is denoted by $(\Omega, \mathfrak{F}, \mathbb{P})$, where $\Omega$ represents the set of all possible outcomes $\omega \in \Omega$, $\mathfrak{F}$ is a $\sigma$-algebra of measurable events, and $\mathbb{P} : \mathfrak{F} \rightarrow [0,1]$ denotes the associated probability measure. For a random variable $v : \Omega \rightarrow H^k(\cD)$, the Bochner space $L^p(\Omega; H^k(\cD))$ is defined as
\begin{equation*}
\begin{split}
L^p(\Omega; H^k(\mathcal{D})) := \Big\{\, v : \Omega \to H^k(\mathcal{D}) \; :\;
&v \text{ is strongly measurable, } \\
&\|v\|_{L^p(\Omega; H^k(\mathcal{D}))} < \infty \,\Big\},
\end{split}
\end{equation*}
where
\[
\|v\|_{L^p(\Omega; H^k(\cD))} =
\begin{cases}
\left( \displaystyle\int_{\Omega} \|v(\omega)\|^p_{k, \cD} \, d\mathbb{P}(\omega) \right)^{1/p}, & \text{for } 1 \leq p < \infty, \\[1.2em]
\displaystyle\esssup_{\omega \in \Omega} \|v(\omega)\|_{k, \cD}, & \text{for } p = \infty.
\end{cases}
\]
Additionally, for a real--valued random field  $\varkappa: \cD \times \Omega \rightarrow \mathbb{R}$ defined on the probability space $(\Omega, \mathfrak{F}, \mathbb{P})$, the expected value (mean) and  covariance functions are defined, respectively, by \x{
\begin{align*}
\mathbb{E}[\varkappa](x) &:= \int_{\Omega} \varkappa(x, \omega) \, d\mathbb{P}(\omega), \qquad x \in \cD, \\
\textnormal{Cov}[\varkappa](x, \widetilde{x}) &:= \int_{\Omega} \big(\varkappa(x, \omega) - \mathbb{E}[\varkappa](x)\big) \big(\varkappa(\widetilde{x}, \omega) - \mathbb{E}[\varkappa](\widetilde{x})\big) \, d\mathbb{P}(\omega), \;\; x, \widetilde{x} \in \cD. 
\end{align*}
Setting $\widetilde{x} = x$, one obtains the variance $\mathbb{V}[\varkappa](x) = \textnormal{Cov}[\varkappa](x, x)$ and the corresponding standard deviation $\kappa_{\varkappa}(x) = \sqrt{\mathbb{V}[\varkappa](x)}$.
}

\subsection{PDE constraint: linear elasticity}

A linear elasticity system that incorporates uncertainties in both the loading and material parameters within the domain $\cW \subset \R^2$ is formulated as follows:
\begin{subequations}\label{eqn:elastic_PDE}
\begin{align}
    -\Div\!\big(A(x,\omega)\, \nabla^s u(x,\omega)\big) &= f(x,\omega) && \text{in } \cW \times \Omega, \label{eqn:elastic_eq}\\
    u(x,\omega) &= 0 && \text{on } \Gamma_d \times \Omega, \label{eqn:dirichlet_bc}\\
    \big(A(x,\omega) \nabla^s u(x,\omega)\big) n &= g(x,\omega) && \text{on } \Gamma_n \times \Omega, \label{eqn:neumann_bc}\\
    \big(A(x,\omega) \nabla^s u(x,\omega)\big) n &= 0 && \text{on } \Gamma_f \times \Omega, \label{eqn:free_bc}
\end{align}
\end{subequations}
where $\Gamma_d$ and $\Gamma_n$ are fixed, disjoint subsets of the boundary $\partial \cW$, representing the Dirichlet and Neumann boundaries, respectively.  The remaining part of the boundary,
\[
\Gamma_f := \partial \cW \setminus (\Gamma_d \cup \Gamma_n),
\]
denotes the \emph{free boundary}, which is the only part subject to optimization during the shape update process. Here, $f: \cW \times \Omega \rightarrow \R^2$ is the body force, $g: \Gamma_n \times \Omega \rightarrow \R^2$ is the traction (Neumann) force, $u: \cW \times \Omega \rightarrow \R^2$ denotes the displacement field, $\nabla^s u := \frac{1}{2}(\nabla u + (\nabla u)^T) : \cW \times \Omega \rightarrow \R^{2 \times 2}$ is the symmetric strain tensor field, $A$ represents the Hooke elasticity tensor,  $A\, \nabla^s u:  \cW \times \Omega \rightarrow \R^{2 \times 2}$ is the stress tensor field, and $n$ denotes the unit outward normal vector to $\partial \cW$. Further, we note that the divergence operator, $\Div$, and the gradient operator, $\nabla$, always refer to differentiation with respect to the spatial variable $x$, unless otherwise stated.

In the optimization process, handling the variable domain $\cW$ poses significant computational challenges. To circumvent these difficulties, the computational problem is reformulated on a fixed reference domain $\cD$ using the \emph{ersatz material} approach \cite{GAllaire_FJouve_AMToader_2004,MYWang_XWang_DGuo_2003a}. In this framework, the strong (solid) material phase is distinguished from the void (or weak) material phase by defining the Hooke elasticity tensor as
\begin{equation}\label{eqn:hook_ersatz}
    A_\cW = A\, \chi_\cW + \epsilon A\, \chi_{\cD \setminus \cW},
\end{equation}
where $\epsilon$ is a prescribed small parameter satisfying $0 < \epsilon \ll 1$, and $\chi_{\cW}$ denotes the characteristic (indicator) function of the domain $\cW$. Here, $A$ is assigned to the strong phase $\cW \subset \cD$,  whereas $\epsilon  A$ is assigned to the weak phase $\cD \setminus \cW$. For any symmetric matrix $\zeta$, Hooke’s law for a linear isotropic elastic material is characterized by
\[
    A(x,\omega)\, \zeta = 2\,\mu(x,\omega)\, \zeta + \lambda(x,\omega)\, \Tr(\zeta)\, \mathrm{Id},
\]
where  $\mu(x,\omega)$ and $\lambda(x,\omega)$ are the Lamé moduli of the material, assumed to depend on both $x \in \cD$ and $\omega \in \Omega$, $\Tr(\cdot)$ denotes the trace operator, and $\mathrm{Id}$ is the $2 \times 2$ identity matrix. In a similar manner, the ersatz material approach is applied to the body force $f$, which is extended to the fixed domain $\cD$. In this approach, the optimization process is still performed with respect to the variable set $\cW$, which is embedded within the fixed larger computational domain $\cD$.  Accordingly, the linear elasticity system is reformulated as
\begin{subequations}\label{eqn:elastic_PDE_ersatz}
\begin{align}
    -\Div\!\big(A_{\cW}(x,\omega)\, \nabla^s u(x,\omega)\big) &= f_{\cW}(x,\omega) && \text{in } \cD \times \Omega, \label{eqn:ersatz_eq}\\
    u(x,\omega) &= 0 && \text{on } \Gamma_d \times \Omega, \label{eqn:ersatz_dir}\\
    \big(A_{\cW}(x,\omega) \nabla^s u(x,\omega)\big) n &= g(x,\omega) && \text{on } \Gamma_n \times \Omega, \label{eqn:ersatz_neu}\\
    \big(A_{\cW}(x,\omega) \nabla^s u(x,\omega)\big) n &= 0 && \text{on } \big(\partial \cD \backslash (\overline \Gamma_d \cup \overline \Gamma_n) \big) \times \Omega, \label{eqn:ersatz_free}
\end{align}
\end{subequations}
where $\cD$ is a domain such that  $\Gamma_d \cup \Gamma_n \subset \partial \cD$ and  the interface between the strong and weak material phases is given by $\partial \cW$, which constitutes the only free (i.e., design-dependent) interface. For sufficiently small values of $\epsilon$, the solution of \eqref{eqn:elastic_PDE_ersatz}  corresponds to an approximation of the solution to \eqref{eqn:elastic_PDE}; see also \cite{MDambrine_DKateb_2010} for discussions on stability and numerical consistency issues.

To ensure the well-posedness of the displacement field  $u: \cD  \times \Omega \ra \R^2$ in \eqref{eqn:elastic_PDE_ersatz},  we require the following assumptions on the given random input data.

\begin{assumption}\label{assump:material_coeff}
Let the Lamé parameters $\lambda(x,\omega)$ and $\mu(x,\omega)$ belong to the space $L^{\infty}(\Omega; L^{\infty}(\cD))$.  
There exist positive constants $\mu_{\textnormal{min}}, \mu_{\textnormal{max}}, \lambda_{\textnormal{min}}, \lambda_{\textnormal{max}} \in (0, \infty)$ such that
\[
    \mu_{\textnormal{min}} \leq \mu(x,\omega) \leq \mu_{\textnormal{max}}
    \quad \text{and} \quad
    \lambda_{\textnormal{min}} \leq 2\mu(x,\omega) + 2\lambda(x,\omega) \leq \lambda_{\textnormal{max}}
\]
for almost every $(x,\omega) \in \cD \times \Omega$.
\end{assumption}

\begin{assumption}\label{assump:force_funcs}
The stochastic body force $f_{\cW}(x,\omega)$ and boundary traction $g(x,\omega)$ possess continuous and bounded covariance functions and satisfy
\[
    f_{\cW} \in L^2(\Omega; L^2(\cD)^2)
    \quad \text{and} \quad
    g \in L^2(\Omega; L^2(\Gamma_n)^2).
\]
\end{assumption}

\noindent Then, introducing  the Hilbert space $H^1_d(\cD)^2 := \{ v \in H^1(\cD)^2: \, v =0  \; \hbox{on} \; \Gamma_d \}$
equipped with the $H^1(\cD)^2$-norm, we state the weak formulation of the elasticity problem \eqref{eqn:elastic_PDE_ersatz}  as follows: find $u \in L^2(\Omega; H^1_d (\cD)^2)$ such that 
\begin{equation}\label{eqn:weak_form_continuous}
    a[u,v] = [f_\cW, v] + [g, v]_{\Gamma_n}
    \qquad \forall  v \in L^2(\Omega; H^1_d(\cD)^2),
\end{equation}
where 
\begin{align*}
    a[u,v] &:= \int_\Omega \int_\cD A_\cW(x,\omega)\, \nabla^s u(x,\omega) : \nabla^s v(x,\omega) \, dx \, d\mathbb{P}(\omega), \\[0.5em]
    [f_\cW, v] &:= \int_\Omega \int_\cD f_\cW(x,\omega) \cdot v(x,\omega) \, dx \, d\mathbb{P}(\omega), \\[0.5em]
    [g, v]_{\Gamma_n} &:= \int_\Omega \int_{\Gamma_n} g(x,\omega) \cdot v(x,\omega) \, ds \, d\mathbb{P}(\omega),
\end{align*}
and the notation $ : $ denotes the Frobenius inner product of tensors. Given the assumptions outlined in Assumptions \ref{assump:material_coeff} and \ref{assump:force_funcs}, the Lax-Milgram theorem guarantees  that the problem described by \eqref{eqn:elastic_PDE_ersatz} is well-posed; see, e.g., \cite{JMFrutos_DHPerez_MKessler_FPeriago_2016}.

\subsection{Objective functional: compliance}

In this study, we consider a robust structural shape optimization problem that accounts for uncertainties in both the loading and material parameters.  
The problem is formulated as
\begin{equation}\label{eqn:objective}
    \min_{\cW \in \mathcal{U}^{\textup{ad}}} \; \E[J(\cW, \omega)]
    \quad \text{subject to} \quad \eqref{eqn:elastic_PDE_ersatz},
\end{equation}
where $J : \mathcal{U}^{\textup{ad}} \times \Omega \rightarrow \R$ denotes the objective functional such that the mapping $\omega \mapsto J(\cW, \omega)$ is $\mathbb{P}$-integrable and well-defined for all admissible domains $\cW \in \mathcal{U}^{\textup{ad}}$,  \x{where
$\mathcal{U}^{\textup{ad}}$ denotes the set of admissible domains contained in $\cD$.}

Before discussing the objective functional and the set of admissible shapes in \eqref{eqn:objective}, we recall the notion of the $\varepsilon$-cone property \cite{DChenais_1975}.

\begin{definition}[$\varepsilon$-cone property]
An open set $\cW \subset \R^2$ is said to satisfy the $\varepsilon$-cone property if, for every point $x \in \partial \cW$, there exists a unit vector $q \in \R^2$  such that
\[
    \mathcal{C}(y, q, \varepsilon) \subset \cW 
    \qquad \forall y \in \overline{\cW} \cap B(x, \varepsilon),
\]
where the open cone $\mathcal{C}(y, q, \varepsilon)$ is defined by
\[
    \mathcal{C}(y, q, \varepsilon)
    := \Big\{ z \in \R^2 : (z - y)\cdot q  \geq |z - y| \cos \varepsilon, \;
       0 < |z - y| < \varepsilon \Big\}.
\]
\end{definition}

\noindent We now define the class of admissible domains $\mathcal{U}^{\textup{ad}}$ as
\begin{equation}\label{eqn:admissible}
    \mathcal{U}^{\textup{ad}}
    = \left\{
        \cW \in \cW_{\varepsilon} :
        \Gamma_d \cup \Gamma_n \subset \partial \cW, \;
        0 < |\cW| = \varsigma\, |\cD|
      \right\},
\end{equation}
where
$
    \cW_{\varepsilon}
    := \left\{
        \cW \subset \cD \; \text{open} :
        \cW \text{ satisfies the } \varepsilon\text{-cone property}
      \right\},
$
and $\varsigma \in (0,1)$ denotes the prescribed volume fraction, that is, the ratio between the domain measure $|\cW| := \int_{\cW} dx$ and the working domain volume $|\cD|$.

For a given domain $\cW \in \mathcal{U}^{\textup{ad}}$ and each realization $\omega \in \Omega$, we define the objective functional $J$ in \eqref{eqn:objective} as the commonly used output functional, namely, the \emph{compliance}
\begin{equation}\label{eqn:compliance}
    J(\cW, \omega)
    = \int_{\cD} f_{\cW}(x,\omega) \cdot u(x,\omega) \, dx
      + \int_{\Gamma_n} g(x,\omega) \cdot u(x,\omega) \, ds,
\end{equation}
where $u(x,\omega)$ denotes the weak solution of the elasticity system \eqref{eqn:elastic_PDE_ersatz} corresponding to the realization $\omega \in \Omega$. Accordingly, the robust shape optimization problem is formulated as
\begin{equation}\label{eqn:optimization_constraint}
    \cW^* = \operatorname*{arg\,min}_{\cW \in \mathcal{U}^{\textup{ad}}}
    \, \mathcal{J}(\cW)
    := \E[J(\cW, \omega)].
\end{equation}
Under the $\varepsilon$-cone assumption on the admissible set $\mathcal{U}^{\textup{ad}}$, the existence of a solution to the optimization problem \eqref{eqn:optimization_constraint} can be established; see, e.g., \cite[Theorem~2.1]{JMFrutos_DHPerez_MKessler_FPeriago_2016}.  
To enforce the volume constraint  in practice, a penalization approach is employed, leading to the following formulation:
\begin{equation}\label{eqn:optimization_penalized}
    \min_{\cW \subset \cD} \;
    \mathcal{J}^P(\cW)
    = \E[J^P(\cW, \omega)],
\end{equation}
where the penalized objective functional is defined as
\vspace{-2mm}
\begin{equation}\label{eqn:optimization_penalized_fixed}
    J^P(\cW, \omega)
    = J(\cW, \omega)
      + \frac{\Lambda}{2}
        \left( \varsigma - \frac{|\cW|}{|\cD|} \right)^{2},
\end{equation}
with $\Lambda > 0$ denoting the penalty parameter.


\section{Stochastic shape derivative}\label{sec:shape_diff}

In the gradient-based minimization of objective functions, we follow the classical notion of Hadamard's boundary variation method (see, e.g., \cite{JSokolowski_JPZolesio_1992a}) to define perturbed shapes from an initial shape $\mathcal{W} \in \mathcal{U}^{\textup{ad}}$.

\x{Let $L \subset \partial\mathcal D$ denote the set of singular boundary points of $\partial\mathcal D$ at which the outward unit normal vector
$n$ is not uniquely defined, such as corners and kinks. For $k \ge 1$, we define the space of admissible deformation fields by
\[
\Theta^k(\mathcal D)
:=
\Bigl\{
\theta \in C^k(\overline{\mathcal D};\mathbb R^2)
:
\theta \cdot n = 0
\ \text{on}\ \partial\mathcal D \setminus L,
\;\;
\theta = 0
\ \text{on}\ L \cup \Gamma_d \cup \Gamma_n
\Bigr\},
\]
equipped with the topology induced by $C^k(\overline{\mathcal D};\mathbb R^2)$. The condition $\theta \cdot n = 0$ on $\partial\mathcal D \setminus L$, together with $\theta = 0$ on $L$, ensures that the flow preserves the hold-all domain. In addition, the condition $\theta = 0$ on $\Gamma_d \cup \Gamma_n$ prevents any deformation of the prescribed boundary portions, so that only the free boundary is allowed to evolve.}

Each vector field $\theta\in\Theta^k(\mathcal D)$ is assumed to admit a $C^k$-extension to an open neighborhood of $\overline{\mathcal D}$,
which, by abuse of notation, is still denoted by $\theta$. For a vector field $\theta\in\Theta^k(\mathcal D)$, let $\Phi_t$ denote the flow generated by
\[
\dot x(t)=\theta(x(t)),
\qquad
x(0)=x_0.
\]
The perturbed domain is then defined by
\[
\mathcal W_t:=\Phi_t(\mathcal W).
\]
For sufficiently small $t>0$, the mapping $\Phi_t$ is a $C^k$-diffeomorphism in a neighborhood of
$\overline{\mathcal D}$, and consequently $\overline{\mathcal W_t}\subset\overline{\mathcal D}$; see \cite[Theorem~2.16]{JSokolowski_JPZolesio_1992a}.

Now we are ready to introduce the definition of the shape derivative for a fixed realization $\omega \in \Omega$.

\begin{definition}\label{defn:shape_derivative}
\x{Let $\cW \subset  \R^2$ be an open set and let $\omega\in\Omega$ be fixed. The Eulerian semiderivative of a functional
$J(\cdot,\omega)$ at $\cW$ with $|\cW|\neq0$ in the direction $\theta\in\Theta^k(\cD)$ is defined, whenever the limit exists, by
\[
\diff J(\cW,\omega)(\theta)
:=
\lim_{t\to0^+}
\frac{
J(\Phi_t(\cW),\omega)
-
J(\cW,\omega)
}{t}.
\]
If, for every direction $\theta\in\Theta^k(\cD)$, the derivative exists and the mapping
\[
\Theta^k(\cD)\ni\theta
\longmapsto
\diff J(\cW,\omega)(\theta)\in\R
\]
is linear and continuous, then the functional $J(\cdot,\omega)$ is said to be shape differentiable at $\cW$.}
\end{definition}

Under the above definition, the objective functional in \eqref{eqn:optimization_penalized}, namely $\mathcal{J}^P(\cW)= \E[J^P(\cW,\omega)]$, is shape differentiable at $\cW$ provided that $J^P(\cdot,\omega)$ is shape differentiable at $\cW$ for almost every $\omega \in \Omega$, and \x{the associated
difference quotient is dominated by a $\mathbb{P}$-integrable function uniformly for $t>0$ sufficiently close to $0$; see, Theorem~\ref{thm:shape_differentiable}.} In this case, the derivative satisfies
\begin{equation*}
    \diff \mathcal{J}^P(\cW)(\theta)
    =
    \E[\diff J^P(\cW,\omega)(\theta)],
    \qquad
    \forall \theta \in \Theta^k(\cD),
\end{equation*}
see, for instance, \cite[Lemma~2.14]{CGeierbach_ELRomero_KWelker_2021}.

Next, we compute the shape derivative of the  penalized objective functional \eqref{eqn:optimization_penalized_fixed}, that is, $J^P(\cW,\omega)$ by employing the averaged-adjoint  and Lagrangian methods introduced in \cite{KSturm_2015} for a fixed but arbitrary realization $\omega \in \Omega$. \x{The following result is  obtained by applying the deterministic shape-calculus framework of \cite[Section~3.3]{ALaurain_2018} to each realization $\omega \in \Omega$ and is included for completeness, as it provides the quantities required for the adaptive error estimation and optimization procedures developed later.}

\begin{proposition}\label{prop:derivative_penalizedobjective}
    For a fixed realization $\omega \in \Omega$, the shape derivative of the penalized objective functional $J^P(\cW,\omega)$ in \eqref{eqn:optimization_penalized_fixed}  is given by
    \begin{equation} \label{eqn:derivative_penalizedobjective}
        \diff J^P(\cW, \omega)(\theta) = \underbrace{\int_\cD \big( S_\omega : \nabla \theta \big) \, dx}_{\diff J(\cW, \omega)(\theta)} - \widetilde{\Lambda} \int_\cW \Div \theta \, dx,
    \end{equation}
    for all $\theta\in\Theta^k(\cD)$. Here,
    \begin{align*}
        S_\omega :=& \, 2 \nabla u(\cdot,\omega)^T A_\cW \nabla^s u(\cdot,\omega) + \big(2f_\cW \cdot u(\cdot,\omega) - A_\cW \nabla^s u(\cdot,\omega) : \nabla^s u(\cdot,\omega) \big)\mathrm{Id}, \\
        \widetilde{\Lambda} :=& \, \frac{\Lambda}{|\cD|} \left(\varsigma  - \frac{|\cW|}{|\cD|}\right),
    \end{align*}
     and  \x{$u(\cdot,\omega)$ denotes the weak solution of \eqref{eqn:elastic_PDE_ersatz} posed on the fixed computational domain $\cD$ corresponding to the design domain $\cW$ and the realization $\omega\in\Omega$.}
\end{proposition}

\noindent {\it Proof} We refer to Appendix~\ref{appendix:derivative_penalizedobjective} for the proof of Proposition~\ref{prop:derivative_penalizedobjective}.
\qed

\begin{remark}
The expression in \eqref{eqn:derivative_penalizedobjective} is commonly referred to as the distributed, volumetric, or domain representation of the shape derivative. 
Under standard regularity assumptions, the shape derivative depends only on the normal component $\theta \cdot n$ on the interface $\partial \cW$; 
see \cite[pp.~480--481]{MCDelfour_JPZolesio_2011} for further details.  Although the boundary (Hadamard) form of the shape derivative is frequently employed 
in the development of level-set based numerical algorithms, we adopt the domain representation in \eqref{eqn:derivative_penalizedobjective} because of the favorable performance of volumetric formulations when combined with finite element discretization (see, e.g., \cite{RHiptmair_APaganini_SSargheini_2015}). \x{Further, since the shape derivative is employed only in its distributed form over the fixed hold-all domain $\cD$, no boundary representation of the shape derivative is required, and therefore no additional smoothness assumptions beyond those already imposed on $\partial\cW$ are needed.}
\end{remark}

We now establish the shape differentiability of the functional $\mathcal{J}^P(\cW)$ in \eqref{eqn:optimization_penalized} 
by adapting the techniques developed in \cite[Theorem~3.6]{CGeierbach_ELRomero_KWelker_2021}, where the Laplace equation  serves as the governing PDE constraint.

\begin{theorem}\label{thm:shape_differentiable}
    \x{The functional $\mathcal J^P$ defined in \eqref{eqn:optimization_penalized} is shape differentiable at every design $\cW$ for which the state problem is well-posed.}
\end{theorem}

\noindent {\it Proof}
Let $u_t(\omega) \in H^1_d(\cD)^2$  be the weak solution of \eqref{eqn:elastic_PDE_ersatz} with $\cW$ replaced by $\cW_t$. In view of Proposition~\ref{prop:derivative_penalizedobjective}, it remains to verify the domination condition of \cite[Lemma~2.14]{CGeierbach_ELRomero_KWelker_2021}, namely,
\x{
\begin{equation*}
        \left | \Psi_t(\omega) \right|= \left| \frac{J^P(\cW_t,\omega) - J^P(\cW,\omega)}{t} \right| \leq R(\omega) \quad \forall  t \in (0, \tau], 
\end{equation*}
where $R:\Omega\to[0,\infty)$ is $\mathbb P$-integrable.}
 
Applying the change of variables $x \mapsto \Phi_t(x)$, which allows the relation $f_{\cW} = f_{\cW_t} \circ \Phi_t$ to be used, we obtain
\begin{align*}
    t\Psi_t(\omega)
    =&
    \int_\cD f_\cW(\omega)\cdot\widetilde{u}_t(\omega)\,\zeta(t)\,dx
    -
    \int_\cD f_\cW(\omega)\cdot u(\omega)\,dx  \\
    &+
    \int_{\Gamma_n}
    g(\omega)\cdot
    \big(\widetilde{u}_t(\omega)-u(\omega)\big)\,ds \\
    &+
    \frac{\Lambda}{2}
    \left[
    \left(
    \varsigma-\frac{1}{|\cD|}
    \int_\cW \zeta(t)\,dx
    \right)^2
    -
    \left(
    \varsigma-\frac{|\cW|}{|\cD|}
    \right)^2
    \right],
\end{align*}
where $\widetilde{u}_t=u_t\circ\Phi_t$ and $\zeta(t)$ denotes the Jacobian determinant of the transformation $x\mapsto\Phi_t(x)$. \x{Since $\zeta(0)=1$ and $\zeta$ is continuously differentiable with respect to $t$, there exist constants $C_\zeta>0$ and $\tau_1>0$ such that
\[
    |\zeta(t)|\le C_\zeta,
    \qquad
    |\zeta(t)-1|\le C_\zeta t,
    \qquad
    \forall t\in[0,\tau_1].
\]
Using these estimates together with the trace inequality, whose constant is denoted by $C_{tr}>0$, yields
\begin{align}\label{eqn:sd_2}
  |t\Psi_t(\omega)|
  \leq& \;
  C_\zeta
  \|f_\cW(\omega)\|_{L^2(\cD)^2}
  \|\widetilde{u}_t(\omega)-u(\omega)\|_{H^1(\cD)^2}
  \nonumber\\
  &+
  C_\zeta t
  \|f_\cW(\omega)\|_{L^2(\cD)^2}
  \|u(\omega)\|_{H^1(\cD)^2}
  \nonumber\\
  &+
  C_{tr}
  \|g(\omega)\|_{L^2(\Gamma_n)^2}
  \|\widetilde{u}_t(\omega)-u(\omega)\|_{H^1(\cD)^2}
  \nonumber\\
  &+
  \frac{\Lambda}{2}
    \,C_\zeta t\,
    \frac{|\cW|}{|\cD|}
    \left(
    2\left|
    \varsigma-\frac{|\cW|}{|\cD|}
    \right|
    +
    C_\zeta t\,\frac{|\cW|}{|\cD|}
    \right).
\end{align}}

\noindent Next, we derive an upper bound for $\|\widetilde{u}_t(\omega) - u(\omega)\|_{H^1(\cD)^2}$. Recall that  $\widetilde{u}_t(\omega) = u_t(\omega) \circ \Phi_t $ satisfies the following variational formulation: 
\begin{equation}\label{eqn:sd_3}
  \int_\cD A_\cW(\omega) \varepsilon(t, \widetilde{u}_t) : \varepsilon(t, v) \, \zeta(t) \, dx = \int_\cD f_{\cW}(\omega) \cdot v \, \zeta(t) \,  dx  +\int_{\Gamma_n} g(\omega) \cdot v \, ds, 
\end{equation}
for all $v\in H^1_d(\cD)^2$. Correspondingly, the associated energy functional  $\mathcal{E}_\omega : [0, \tau) \times H^1_d(\cD)^2 \ra \R$ for \eqref{eqn:sd_3} is  defined by
\begin{equation*}
        \mathcal{E}_\omega(t, \varphi) = \frac{1}{2} \int_\cD A_\cW(\omega) \varepsilon(t,\varphi):\varepsilon(t,\varphi) \, \zeta(t) \,  dx - \int_\cD f_\cW(\omega) \cdot \varphi \, \zeta(t) \, dx - \int_{\Gamma_n} g(\omega) \cdot \varphi \, ds.
\end{equation*}
The first- and second-order Gâteaux derivatives of $\mathcal{E}_\omega$ with respect to $\varphi$ in the directions 
$\psi, \psi_0 \in H^1_d(\cD)^2$ are given, respectively, by
\begin{align*}
    (\delta_\varphi \mathcal{E}_\omega, \psi)
    =& \int_\cD 
    \big( A_\cW(\omega)\, \varepsilon(t, \varphi) : \varepsilon(t, \psi)
    - f_\cW(\omega) \cdot \psi \big)\, \zeta(t)\, dx
    - \int_{\Gamma_n} g(\omega) \cdot \psi\, ds, \\[0.5em]
    (\delta_\varphi^2 \mathcal{E}_\omega, \psi, \psi_0)
    =& \int_\cD 
    A_\cW(\omega)\, \varepsilon(t, \psi) : \varepsilon(t, \psi_0)\, 
    \zeta(t)\, dx.
\end{align*}
An application of Assumption~\ref{assump:material_coeff}, \x{the uniform positivity of $\zeta(t)$ for sufficiently small $t$,
and Korn's inequality implies that there exist constants $c_1>0$ and $\tau_2>0$, independent of $t$ and $\omega$, such that}
\begin{equation}\label{eqn:sd_4}
    (\delta_\varphi^2 \mathcal{E}_\omega,\psi,\psi)
    \ge c_1 \|\psi\|_{H^1(\cD)^2}^2
    \qquad \forall \psi \in H^1_d(\cD)^2,
\end{equation}
for all $t\in[0,\tau_2]$. Therefore, $\mathcal E_\omega(t,\cdot)$ is strongly convex on
$H^1_d(\cD)^2$ for all $t\in[0,\tau_2]$. Consequently, the minimization problem
\[
   \min_{\varphi \in H^1_d(\cD)^2} \mathcal E_\omega(t,\varphi)
\]
admits a unique minimizer for every $t\in[0,\tau_2]$. Moreover, this minimizer coincides with the weak solution $\widetilde{u}_t(\omega)$ of \eqref{eqn:sd_3} and satisfies the first-order optimality condition
\[
    (\delta_\varphi \mathcal{E}_\omega(t,\widetilde{u}_t),\psi)
    = 0
    \qquad
    \forall \psi \in H^1_d(\cD)^2.
\]
For $\widetilde{u}_t^r = r u + (1-r)\widetilde{u}_t$ with $r \in [0,1]$, it follows that
\begin{align*}
   &\int_0^1
   \left(
   \delta_\varphi^2 \mathcal{E}_\omega(t,\widetilde{u}_t^r),
   \widetilde{u}_t-u,
   \widetilde{u}_t-u
   \right)\,dr \\
   &\qquad=
   \left(
   \delta_\varphi \mathcal{E}_\omega(t,u),
   \widetilde{u}_t-u
   \right)
   -
   \left(
   \delta_\varphi \mathcal{E}_\omega(t,\widetilde{u}_t),
   \widetilde{u}_t-u
   \right)
   \nonumber\\
   &\qquad=
   \left(
   \delta_\varphi \mathcal{E}_\omega(t,u),
   \widetilde{u}_t-u
   \right)
   -
   \left(
   \delta_\varphi \mathcal{E}_\omega(0,u),
   \widetilde{u}_t-u
   \right)
   \nonumber\\
   &\qquad=
   t
   \left(
   \delta_{t,\varphi}\mathcal{E}_\omega(t r_t,u),
   \widetilde{u}_t-u
   \right),
\end{align*}
where the second equality follows from the first-order optimality conditions
\[
(\delta_\varphi \mathcal E_\omega(t,\widetilde u_t),\psi)
=0, \qquad 
(\delta_\varphi \mathcal E_\omega(0,u),\psi)
=
0
\qquad
\forall \psi\in H^1_d(\cD)^2.
\]
Moreover, by the mean value theorem, there exists $r_t\in(0,1)$ such that the last equality holds. \x{Here, we note that $\delta_{t,\varphi}\mathcal E_\omega$
denotes the derivative of $\delta_\varphi\mathcal E_\omega$ with respect to the parameter $t$. By Assumption~\ref{assump:material_coeff}, the continuous differentiability of $\zeta(t)$, and the smooth dependence of the transformation $\Phi_t$ on $t$, there exist $\tau_3>0$ and $c_2>0$, independent of $t$ and $\omega$, such that}
\begin{equation}\label{eqn:sd_6}
    \left|
    \left(
    \delta_{t,\varphi}\mathcal{E}_\omega(t,\varphi),
    \psi
    \right)
    \right|
    \le
    c_2
    \left(
    \|\varphi\|_{H^1(\cD)^2}
    +
    \|f_\cW(\omega)\|_{L^2(\cD)^2}
    \right)
    \|\psi\|_{H^1(\cD)^2}
\end{equation}
for all $t\in[0,\tau_3]$. Using \eqref{eqn:sd_6} together with \eqref{eqn:sd_4}, we obtain
\begin{equation}\label{eqn:sd_7}
    \|\widetilde{u}_t(\omega)-u(\omega)\|_{H^1(\cD)^2}
    \le
    c\,t\,
    \left(
    \|u(\omega)\|_{H^1(\cD)^2}
    +
    \|f_\cW(\omega)\|_{L^2(\cD)^2}
    \right),
\end{equation}
\x{where $c:=c_2/c_1$, for all $t\in[0,\min(\tau_2,\tau_3)]$. Inserting \eqref{eqn:sd_7}  into \eqref{eqn:sd_2}, and using $t\le \tau$, we obtain
\begin{align*}
    |t\,\Psi_t(\omega)|
    &\le
    t
    \Big[
    C_\zeta c
    \|f_\cW(\omega)\|_{L^2(\cD)^2}
    \Big(
    \|u(\omega)\|_{H^1(\cD)^2}
    +
    \|f_\cW(\omega)\|_{L^2(\cD)^2}
    \Big)
    \\
    &\qquad
    +
    C_\zeta
    \|f_\cW(\omega)\|_{L^2(\cD)^2}
    \|u(\omega)\|_{H^1(\cD)^2}
    \\
    &\qquad
    +
    C_{tr}c
    \|g(\omega)\|_{L^2(\Gamma_n)^2}
    \Big(
    \|u(\omega)\|_{H^1(\cD)^2}
    +
    \|f_\cW(\omega)\|_{L^2(\cD)^2}
    \Big)
    \\
    &\qquad
    +
    \frac{\Lambda}{2}
    C_\zeta
    \frac{|\cW|}{|\cD|}
    \left(
    2\left|
    \varsigma-\frac{|\cW|}{|\cD|}
    \right|
    +
    C_\zeta \tau \frac{|\cW|}{|\cD|}
    \right)
    \Big].
\end{align*}
From the stability estimate, there exists a constant $C_R>0$,
independent of $t$ and $\omega$, such that
\[
    |t\,\Psi_t(\omega)|
    \le
    t\,R(\omega),
\]
where
\[
\begin{aligned}
R(\omega)
:={}&
C_R
\left(
\|f_\cW(\omega)\|_{L^2(\cD)^2}
+
\|g(\omega)\|_{L^2(\Gamma_n)^2}
\right)^2
\\
&+
\frac{\Lambda}{2}
C_\zeta
\frac{|\cW|}{|\cD|}
\left(
2\left|
\varsigma-\frac{|\cW|}{|\cD|}
\right|
+
C_\zeta \tau \frac{|\cW|}{|\cD|}
\right).
\end{aligned}
\]
Thus,
\[
    |\Psi_t(\omega)|
    \le
    R(\omega)
    \qquad
    \forall t\in(0,\tau],
\]
with
\[
    \tau:=\min\{\tau_1,\tau_2,\tau_3\}.
\]
Moreover, $R\in L^1(\Omega)$ by Assumption~\ref{assump:force_funcs}. Hence, the assumptions of \cite[Lemma~2.14]{CGeierbach_ELRomero_KWelker_2021} are satisfied, and the proof is complete.}
\qed

\section{Numerical approximation techniques} \label{sec:discretization} 

In this section, we present numerical approximation techniques that reduce the infinite-dimensional optimization problem to a finite-dimensional one. To alleviate the computational burden associated with the number of samples and the spatial degrees of freedom, we introduce a posteriori error estimators for both settings. Additionally, the method for determining a descent direction from the shape derivative is explored and incorporated into the level-set method. Further, the optimization step length in the (mini-batch) stochastic gradient method is  adjusted adaptively using an estimate of the Lipschitz constant.

\subsection{Representation of random fields}

To obtain a numerical approximation of the problem described in \eqref{eqn:elastic_PDE_ersatz}, we adopt the finite-dimensional noise assumption as introduced by Wiener in \cite{NWiener_1938a}.

\begin{assumption} \label{asm:finite_noise}
There exists an $M$-dimensional random vector \x{$\xi: \, \Omega \rightarrow \Xi=\prod \limits_{i=1}^M \Xi_i \subset \R^M$} with  joint  probability density function $\rho: \, \Xi \rightarrow \mathbb{R}_+$ such that 
$z(\cdot,\omega) \equiv z_M (\cdot, \xi(\omega) )$, where   $z \in \{f_\cW, g, \mu, \lambda \}$.
\end{assumption} 

Assumption~\ref{asm:finite_noise} can be realized through various finite-dimensional representations of random fields. One common approach is the truncated Karhunen--Lo\`{e}ve (KL) expansion \cite{KKarhunen_1947,MLoeve_1946}, given by
\begin{equation}\label{eqn:KLexpansion}
  \varkappa_M(x, \omega) = \overline{\varkappa}(x) + \kappa_{\varkappa} \sum \limits_{k=1}^{M} \sqrt{\lambda_k}b_k(x)\xi_k(\omega),
\end{equation}
where $\{\xi_k(\omega)\}_{k=1}^{M}$ are mutually uncorrelated random variables with zero mean and unit variance, i.e., $\mathbb{E}[\xi_k]=0$ and $\mathbb{E}[\xi_i \xi_j]=\delta_{ij}$. The functions $\overline{\varkappa}$ and $\kappa_{\varkappa}$ denote the mean and standard deviation of the random field $\varkappa$, respectively. The eigenpairs $\{(\lambda_k,b_k)\}_{k=1}^{M}$ are obtained by solving the Fredholm integral equation associated with the covariance kernel $\mathcal{C}_{\varkappa}$, yielding a sequence of nonincreasing positive eigenvalues $\lambda_1 \ge \lambda_2 \ge \cdots > 0$ together with the corresponding  orthonormal eigenfunctions. The truncated KL expansions are further chosen so that the positivity and boundedness conditions of Assumption \ref{assump:material_coeff} are preserved.

\subsection{Descent direction with level-set method}\label{subsec:descent_direction}

In the gradient-based optimization process, for a given realization $\xi \in \Xi$, we seek a descent direction $\theta$ satisfying
\[
\diff J^P(\cW,\xi)(\theta) < 0.
\]
To this end, \x{by introducing $
X
:=
\Bigl\{
\theta\in H^1(\cD)^2:
\theta\cdot n=0 \text{ on }\partial\cD\setminus L,
\;
\theta=0 \text{ on }L\cup\Gamma_d\cup\Gamma_n
\Bigr\}$, we solve the deformation equation (see, e.g., \cite{ALaurain_2018,VSchulz_MSiebenborn_2016}): find $\theta\in X$ such that
\begin{equation}\label{eqn:weak_form_theta}
    b(\theta,\varphi)
    =
    -\diff J^P(\cW,\xi)(\varphi)
    \qquad
    \forall \varphi\in X,
\end{equation}
where the positive-definite bilinear form
$b:X\times X\to\mathbb R$ is defined by
\[
    b(\theta,\varphi)
    =
    \int_\cD
    \big(
    \tau_1 \nabla\theta:\nabla\varphi
    +
    \tau_2 \theta\cdot\varphi
    \big)\,dx,
\]
with constants $\tau_1,\tau_2>0$. For a fixed realization $\xi\in\Xi$, the weak solution of
\eqref{eqn:weak_form_theta} is denoted by $\theta(\xi)\in X$.
Choosing $\varphi=\theta(\xi)$ in \eqref{eqn:weak_form_theta} gives
\[
    \diff J^P(\cW,\xi)(\theta(\xi))
    =
    -b(\theta(\xi),\theta(\xi))
    <0,
\]
provided that $\theta(\xi)\neq 0$. Hence, the deformation field $\theta(\xi)$ defines a strict descent direction for the penalized
objective functional $J^P(\cW,\xi)$.}

\begin{remark}
\x{The deformation bilinear form \eqref{eqn:weak_form_theta} is employed as a numerical mechanism for constructing a descent direction from the shape derivative. This approach follows the Hilbertian extension--regularization framework commonly used in shape optimization; see, e.g., \cite[Section~5.2.2]{GAllaire_CDapogny_FJouve_2021}. While the shape derivative is defined on the admissible perturbation space $\Theta^k(\cD)$, the auxiliary deformation problem is solved in the computational space $X$, which incorporates the boundary restrictions required for the numerical shape update. Thus, $X$ serves as the computational space in which a regularized deformation field is obtained from the shape derivative for use in the numerical optimization procedure.}

\x{We also note that this construction is not unique. As discussed in \cite[Section~5.2.2]{GAllaire_CDapogny_FJouve_2021}, shape updates may also be obtained from alternative procedures, including boundary-velocity formulations, projected gradient methods, and other regularization strategies. The particular choice affects the computed deformation field and the resulting shape evolution, but not the underlying shape derivative itself.}
\end{remark}

Once the deformation field $\theta$ has been computed, the shape $\mathcal W$ can be updated by deforming it along this descent direction. However, directly tracking the evolution of the domain boundary under such perturbations can be challenging. To overcome this difficulty, we employ the level-set method introduced in \cite{SOsher_JASethian_1988}, which provides an implicit representation of the geometry. This formulation greatly simplifies the treatment of topology changes, allowing for the natural creation and merging of holes during the shape evolution; see \cite{NPVDijk_KMaute_MLangelaar_FVKeulen_2013}. In the level-set framework, the evolving domain $\mathcal W_t\subset\mathcal D$ over a fictitious time interval $t\in(0,T)$ is characterized by a Lipschitz continuous function
$\psi:\mathcal D\times(0,T)\to\mathbb R$, referred to as the \emph{level-set function}, defined by
\[
\left\{
\begin{array}{lll}
    \psi(x,t)<0 & \Leftrightarrow & x\in\mathcal W_t,\\
    \psi(x,t)=0 & \Leftrightarrow & x\in\partial\mathcal W_t\cap\mathcal D,\\
    \psi(x,t)>0 & \Leftrightarrow & x\in\mathcal D\setminus\overline{\mathcal W_t}.
\end{array}
\right.
\]
Thus, the boundary $\partial\mathcal W_t$ is represented by the zero level set of $\psi(\cdot,t)$.

Let $x(t) = \Phi_t(x_0)$ denote the trajectory of a point on $\partial \mathcal{W}_t$ originating from the initial domain
$\mathcal{W}$, where the motion is governed by the velocity field $\dot{x}(t) = \theta(x(t))$ with initial condition $x(0) = x_0$. By differentiating
the identity $\psi(x(t),t) = 0$ with respect to $t$, we obtain the following Hamilton--Jacobi type equation (HJE), which describes the evolution of the
shape during the optimization process
\begin{equation}\label{eqn:hje}
    \frac{\partial \psi}{\partial t}
    + \theta \cdot \nabla \psi = 0
    \qquad \text{in } \mathcal{D} \times (0,T).
\end{equation}
Consequently, solving \eqref{eqn:hje} with a prescribed initial level-set function $\psi(\cdot,0)$ corresponds to evolving the
boundary of $\mathcal W_t$ according to \x{the descent velocity field $\theta$} associated with a fixed realization $\xi \in \Xi$.

To keep the computational cost manageable, we solve the Hamilton--Jacobi equation \eqref{eqn:hje} using the Lax-Friedrichs finite difference  scheme
(see \cite{ALaurain_2018,SOsher_CWShu_1991}) on the rectangular domain $\mathcal{D} = [0,l_x] \times [0,l_y]$. For  the nodal values $\psi_{i,j}$  and the deformation field $\theta = (\theta_x, \theta_y)$, the Lax-Friedrichs numerical scheme is given by
\begin{equation*}
    \widehat{H}^{LF}(p^-, p^+, q^-, q^+) = \frac{\theta_x}{2} (p^+ + p^-) + \frac{\theta_y}{2} (q^+ + q^-) - \frac{|\theta_x|}{2} (p^+ - p^-) - \frac{|\theta_y|}{2} (q^+ - q^-),
\end{equation*}
where the finite difference approximations are
\begin{equation*}
    \begin{split}
        p^- &= \partial_x^- \psi_{i,j} \approx \frac{\psi_{i,j} - \psi_{i-1,j}}{\Delta x}, \qquad p^+ = \partial_x^+ \psi_{i,j} \approx \frac{\psi_{i+1,j} - \psi_{i,j}}{\Delta x}, \\
        q^- &= \partial_y^- \psi_{i,j} \approx \frac{\psi_{i,j} - \psi_{i,j-1}}{\Delta y}, \qquad q^+ = \partial_y^+ \psi_{i,j} \approx \frac{\psi_{i,j+1} - \psi_{i,j}}{\Delta y}.
    \end{split}
\end{equation*}
Applying a forward Euler discretization in the fictitious time variable with time step $\Delta t$, the fully discrete scheme corresponding to \eqref{eqn:hje} is given by
\begin{equation}\label{eqn:HJE_discrete}
    \psi_{i,j}^{k+1} = \psi_{i,j}^k - \Delta t \widehat{H}^{LF}(p^-, p^+, q^-, q^+).
\end{equation}
Here, the time step $\Delta t$ is chosen according to the Courant–Friedrichs–Lewy (CFL) condition and is given by 
\[
    \Delta t
    = \alpha_k \frac{\min(\Delta x,\Delta y)}{\theta_{\max}},
\]
where $\alpha_k\in(0,1]$ denotes the step length used in the gradient-based optimization algorithm, $\Delta x = l_x / n_x$ and $\Delta y = l_y / n_y$ are the 
grid spacings in the $x$- and $y$-directions, and $\theta_{\max} = \|\theta\|_\infty$. Additionally, the number of fictitious time steps performed at each
optimization iteration $k$ is adjusted adaptively according to the observed reduction in the objective value $J^P(\mathcal W_k,\xi)$.


\begin{remark}
To prevent the level-set function $\psi$ from becoming excessively flat or overly steep during its evolution, one may employ the reinitialization procedure introduced in \cite{DLChoop_1993}. This procedure restores $\psi$ to a signed-distance function while preserving its zero level-set. For further implementation details and practical considerations, we refer the reader to \cite{GAllaire_FJouve_AMToader_2004,ALaurain_2018}.
\end{remark}

\subsection{Approximation in the parametric space} \label{subsec:approx_parametric}

To approximate statistical moments, such as  the expected value or higher-order moments of a general quantity of interest like $J^P(\cW,\omega)$ in \eqref{eqn:optimization_penalized_fixed}, we employ a  sample-based Monte Carlo approximation over the parametric space $\Xi$. This leads to
\begin{align} \label{eqn:objective_mc_estimator}
&\mathbb{E}[J^P(\cW,\omega)] \approx  \Em[J^P(\cW,\xi)] =  \frac{1}{N} \sum \limits_{i=1}^{N} J^P(\cW,\xi_i) \nonumber \\
& =   \frac{1}{N} \sum_{i=1}^N \left[\int_\cD f_{\cW,M} \cdot u(x, \xi_i) \, dx + \int_{\Gamma_n} g_M \cdot u(x,\xi_i) \, ds\right] +  \frac{\Lambda}{2} \left(\varsigma  - \frac{|\cW|}{|\cD|}\right)^2,
\end{align}
where $\{\xi_i\}_{i=1}^N \subset \Xi$ denotes  a set of $N$ independent samples. Similarly, the stochastic shape derivative expressed in \eqref{eqn:derivative_penalizedobjective} is approximated by
\begin{align}\label{eqn:derivative_mc_estimator}
    &\E[\diff J^P(\cW, \omega)(\varphi)] \nonumber \\
    &\quad \approx \Em[\diff J^P(\cW, \xi)(\varphi)] 
     = \frac{1}{N} \sum_{i=1}^N \left[ \int_\cD \big( S_{\xi_i} : \nabla \varphi  \big) \, dx \right] - \widetilde{\Lambda} \int_\cW \Div \varphi \, dx.
\end{align}

\subsection{Approximation in the spatial domain}\label{subsec:approx_spatial}

Let $\{\mathcal T_h\}_h$ be a quasi-uniform family of triangulations of $\mathcal D$ such that $\overline{\mathcal D} = \bigcup_{K \in \mathcal T_h} \overline{K}$. We define the mesh size by $h = \max_{K \in \mathcal T_h} h_K$, where $h_K$ denotes the diameter of the element $K$. Moreover, let $\mathcal E_h$ denote the partition of the boundary $\partial \mathcal D$ induced by $\mathcal T_h$, meaning that each $E \in \mathcal E_h$ with $E \subset \partial \mathcal D$ is an edge of some element $K \in \mathcal T_h$. Denoting the set of  all linear polynomials on $K$ by $\mathbb{P}_1(K)$,  we  introduce the finite element spaces
\[
V_h
:=
\Bigl\{
v \in C(\overline{\cD})^2 :
v|_K \in [\mathbb P_1(K)]^2
\; \forall K \in \cT_h,
\;
v = 0 \text{ on } \Gamma_d
\Bigr\}
\subset H_d^1(\cD)^2
\]
and
\x{
\[
X_h
:=
\left\{
\theta \in C(\overline{\cD})^2 :
\begin{aligned}
&\theta|_K \in [\mathbb P_1(K)]^2
&& \forall K \in \cT_h,
\\
&\theta \cdot n = 0
&& \text{on } \partial\cD \setminus L,
\\
&\theta = 0
&& \text{on } L \cup \Gamma_d \cup \Gamma_n
\end{aligned}
\right\}
\subset X.
\]}

\noindent After applying the KL expansion to the random fields and approximating the expectation by the Monte Carlo method, the discrete weak formulation of the system
\eqref{eqn:elastic_PDE_ersatz} for a fixed realization $\xi_i\in\Xi$  is as follows: find $u_h(\cdot,\xi_i)\in V_h$  such that 
\begin{equation}\label{eqn:weak_form_discrete}
    a(u_h, v) = l(v)
    \qquad \forall v \in V_h,
\end{equation}
where
\begin{align*}
    a(u,v)
    = &\;
    \int_{\mathcal D}
    A_{\mathcal W,M}(\cdot,\xi_i)\,
    \nabla^s u : \nabla^s v
    \,dx, \\
    l(v)
    = &\;
    \int_{\mathcal D}
    f_{\mathcal W,M}(\cdot,\xi_i)\cdot v
    \,dx
    +
    \int_{\Gamma_n}
    g_M(\cdot,\xi_i)\cdot v
    \,ds.
\end{align*}
We refer the reader to \cite[Chapter~11]{SCBrenner_LRScott_2008} for the existence and uniqueness of the corresponding Galerkin solution.

To generate a sequence of meshes $\{\cT_h^k\}_{k\ge 0}$ adaptively, we employ the dual-weighted residual (DWR) framework \cite{WBangerth_RRannacher_2003a,RBecker_RRannacher_2001a}. In this goal-oriented mesh adaptation strategy, the objective is to control the error in a quantity of interest (QoI) $\cQ:W\to\R$ for some Hilbert space $W$:
\[
    |\cQ(u)-\cQ(u_h)|
    \leq 
    \eta
    \le
    \texttt{tol},
\]
where $\texttt{tol}>0$ is a prescribed tolerance. For completeness, we recall the corresponding error representation formula; see \cite{WBangerth_RRannacher_2003a,RBecker_RRannacher_2001a} for its proof.

\begin{proposition}\label{prop:error_representation}
Let $W_h\subset W$ be a finite-dimensional subspace of a Hilbert space
$W$. Let $u\in W$ and $u_h\in W_h$ denote the exact and discrete
solutions, respectively. Then, for a sufficiently smooth goal functional
$\cQ$, the following a posteriori error representation holds:
\[
    \cQ(u)-\cQ(u_h)
    =
    \frac12 r(z-I_h z)
    +
    \frac12 r^*(u-I_h u)
    +
    E,
\]
where $r$ and $r^*$ denote the residuals of the primal and dual
problems, respectively, $I_hu,I_hz\in W_h$ are suitable interpolants or
projections, and $E$ denotes a higher-order remainder term.

Moreover, if the governing problem is linear and $\cQ$ is linear, the
higher-order remainder vanishes and the error representation simplifies to
\begin{equation}\label{prop:error_representation_linear}
    \cQ(u)-\cQ(u_h)
    =
    r(z-I_h z).
\end{equation}
\end{proposition}

To obtain optimal designs that are comparable to those computed on highly refined meshes, we account not only for the discretization error associated with the constraint PDE, namely, the linear elasticity system \eqref{eqn:elastic_PDE_ersatz}, but also for the discretization error arising from the deformation equation \eqref{eqn:weak_form_theta}, which determines the descent direction. Consequently, multiple goal functionals must be introduced to quantify these distinct sources of discretization error.

Our first goal functional is the compliance functional. For a fixed realization $\xi\in\Xi$, it is defined by
\[
    \mathcal Q^c(u)
    :=
    \int_{\mathcal D}
    f_{\mathcal W,M}\cdot u\,dx
    +
    \int_{\Gamma_n}
    g_M\cdot u\,ds.
\]
For $v \in H_d^1(\cD)^2$,  the primal and dual residuals are given by
\begin{equation*}
    r(v) = l(v) - a(u_h, v) \quad \hbox{and} \quad r^*(v) = \langle \delta_u \cQ^{c}(u_h), v \rangle - a^*(z_h, v),
\end{equation*}
respectively, for the primal solution $u_h$ and the dual solution $z_h$, and $z_h$  solves 
\begin{equation} \label{eqn:weak_form_discrete_adjoint}
      a^*(z_h, v) = \langle \delta_u \cQ^{c}(u_h), v \rangle \qquad \forall  v \in V_h.
\end{equation}
Here, $\langle \delta_u\mathcal Q^c(u_h),v \rangle$ denotes the Gâteaux derivative of $\mathcal Q^c$ at $u_h$ in the direction $v$. In view of the discrete weak form \eqref{eqn:weak_form_discrete} and \eqref{eqn:weak_form_discrete_adjoint}, the primal and dual residuals satisfy the Galerkin orthogonality relations
\begin{equation*}
   r(v_h) = r^*(v_h) = 0 \qquad \forall  v_h \in V_h.
\end{equation*}
By applying integration by parts elementwise to the residuals $r$ and $r^*$ and using the Cauchy--Schwarz inequality, we arrive at the estimate
\begin{equation} 
        |\cQ^{c}(u) - \cQ^{c}(u_h)| \leq \sum_{K \in \mathcal{T}_h} \Big(\underbrace{\rho_K^u \omega_K^z}_{\eta_K^u} + \underbrace{\rho_K^z \omega_K^u}_{\eta_K^z} \Big):= \eta^{c},
\end{equation}
where the element residuals $\rho_K^z, \rho_K^u$ are 
\begin{eqnarray*}
    \rho_K^u &=& \| f_{\cW,M} + \Div(A_{\cW,M} \nabla^s u_h) \|_K + h^{-1/2}_K \| \ljp (A_{\cW,M} \nabla^s u_h) n \rjp \|_{\partial K}, \\
    \rho_K^z &=& \| \delta_u \cQ^{c}(u_h) + \Div(A_{\cW,M} \nabla^s z_h) \|_K + h^{-1/2}_K \| \ljp (A_{\cW,M} \nabla^s z_h) n \rjp \|_{\partial K}, 
\end{eqnarray*}
and the corresponding local weights $\omega_K^u$ and $\omega_K^z$ are given by
\begin{equation*}
    \omega_K^u = \| u - I_h u \|_K + h^{1/2}_K \| u - I_h u \|_{\partial K}, \quad
    \omega_K^z = \| z - I_h z \|_K + h^{1/2}_K \| z - I_h z \|_{\partial K},
\end{equation*}
where $\ljp \cdot \rjp$ denotes the jump residual. For an interior edge
$E=K_1\cap K_2$ with $K_1,K_2\in\cT_h$, it is defined by
\[
\ljp (A_{\cW,M}\nabla^s u_h)n \rjp |_E
=
\frac12
\Big(
(A_{\cW,M}\nabla^s u_h)n|_{K_1}
-
(A_{\cW,M}\nabla^s u_h)n|_{K_2}
\Big),
\]
where $n$ denotes the unit normal vector on $\partial K_1$. On the
Neumann boundary $\Gamma_n$, the jump residual is defined by
\[
\ljp (A_{\cW,M}\nabla^s u_h)n \rjp
=
g_M-(A_{\cW,M}\nabla^s u_h)n,
\]
whereas on $\partial \cD \backslash (\overline \Gamma_d \cup \overline \Gamma_n)$ it is given by
\[
\ljp (A_{\cW,M}\nabla^s u_h)n \rjp
=
-(A_{\cW,M}\nabla^s u_h)n.
\]

\begin{remark}
We note that the underlying quantity of interest $\cQ^c$ is a linear functional. Therefore, the error representation simplifies, and in our numerical simulations we employ the bound
\begin{equation} \label{eqn:estimate_linear}
      |\cQ^{c}(u) - \cQ^{c}(u_h)|  \leq \sum_{K \in \cT_h} \eta_K^{c}  = \sum_{K \in \cT_h} \rho_K^u \omega_K^z,
\end{equation}
see Proposition~\ref{prop:error_representation}. 
\end{remark} 

Motivated by the role of the shape derivative $\diff J^P(\mathcal{W})$ in updating the  geometry, we introduce a second
goal functional to control the discretization error in the deformation field. Specifically, we define
\[
    \cQ^{d}(\theta)
    = -\,\diff J^P(\mathcal{W},\xi)(\theta),
\]
which measures the accuracy of the deformation field used to evolve the shape. Applying \eqref{prop:error_representation_linear}, together with the symmetry of the bilinear form $b(\theta,\varphi)$ in \eqref{eqn:weak_form_theta} and the linearity of $\mathcal Q^d$, yields
\begin{align}\label{eqn:QoI_shape_1}
   \cQ^{d}(\theta) - \cQ^{d}(\theta_h) &= -\diff J^P(\cW,\xi)(\vartheta-I_h \vartheta) - b(\theta_h, \vartheta-I_h \vartheta) \nonumber \\
     &=: r^{d}(\vartheta-I_h \vartheta), \qquad  \forall  I_h \vartheta \in X_h,
\end{align}
where $\vartheta\in X$ denotes the solution of the dual problem associated with \eqref{eqn:weak_form_theta}, and
$I_h\vartheta\in X_h$ denotes an interpolant of $\vartheta$. The residual is given by
\begin{align*}
 r^{d}(\vartheta-I_h \vartheta) &= - \int_\cD \big( S_\xi : \nabla (\vartheta-I_h \vartheta)  \big) \, dx  + \widetilde{\Lambda} \,\int_\cW \Div(\vartheta-I_h \vartheta) \, dx  \\
  &\quad - \int_\cD (\tau_1 \nabla \theta_h : \nabla(\vartheta-I_h \vartheta) + \tau_2 \theta_h \cdot (\vartheta-I_h \vartheta) ) \, dx.
\end{align*}
Here, $S_\xi$ depends on the discrete state $u_h$, that is,
$S_\xi:=S_\xi(u_h)$, since the discrete primal solution is used in the deformation variational problem \eqref{eqn:weak_form_theta}. Applying elementwise integration by \x{parts} over each $K\in\mathcal T_h$, we obtain
\begin{align*}
r^{d}(\vartheta - I_h \vartheta)
&= \sum_{K \in \mathcal{T}_h} \int_{\partial K}
    \ljp -\tau_1 \nabla \theta_h \, n \rjp
    \cdot (\vartheta - I_h \vartheta) \, ds \notag \\
&\quad + \sum_{K \in \mathcal{T}_h} \int_K
    \Big(
        \big( \widetilde{\Lambda} \, \chi_{\mathcal{W}} \, \mathrm{Id}
        - S_\xi \big)
        : \nabla (\vartheta - I_h \vartheta)  \notag \\
&\qquad\qquad
        + (\tau_1 \Div(\nabla \theta_h)
        - \tau_2 \theta_h)
        \cdot (\vartheta - I_h \vartheta)
    \Big) \, dx.
\end{align*}
Then, applying the divergence theorem to the tensor $\widetilde{\Lambda}\,\chi_{\cW}\,\mathrm{Id}-S_\xi$ elementwise on each $K\in\mathcal T_h$ yields
\begin{align}\label{eqn:QoI_shape_3}
        r^{d}(\vartheta-I_h \vartheta)
        =&\sum_{K \in \mathcal{T}_h}
        \int_{\partial K}
        \ljp -(\tau_1 \nabla\theta_h + S_\xi
        - \widetilde{\Lambda}\chi_\cW\mathrm{Id}) n \rjp
        \cdot (\vartheta-I_h \vartheta) \, ds
        \nonumber\\
        &+
        \sum_{K \in \mathcal{T}_h}
        \int_K
        \Big(
        \Div(\tau_1 \nabla\theta_h + S_\xi)
        - \tau_2 \theta_h
        \Big)
        \cdot (\vartheta-I_h \vartheta) \, dx .
\end{align}
Inserting \eqref{eqn:QoI_shape_3} into \eqref{eqn:QoI_shape_1}, we obtain  a dual-weighted residual error indicator
\begin{equation} \label{eqn:estimate_shape}
    |\cQ^{d}(\theta) - \cQ^{d}(\theta_h)| \leq \sum_{K \in \mathcal{T}_h} \rho_K^{\theta} \omega_K^{\vartheta}:= \eta^{d},
\end{equation}
where the element residual  $\rho_K^{\theta}$ and the local weight $\omega_K^{\vartheta}$ are given, respectively, by
\begin{align*}
    \rho_K^{\theta} =& \|  \Div(\tau_1 \nabla\theta_h + S_\xi) - \tau_2 \theta_h   \|_{K} \nonumber \\
     &+  h^{-1/2}_K \| \ljp -(\tau_1 \nabla\theta_h + S_\xi - \widetilde{\Lambda} \, \chi_\cW \,  \mathrm{Id}) n \rjp \|_{\partial K}, \\
    \omega_K^{\vartheta} =&   \| \vartheta-I_h \vartheta \|_K + h^{1/2}_K \| \vartheta-I_h \vartheta  \|_{\partial K}.
\end{align*}

\x{To obtain fully computable a posteriori error estimators from \eqref{eqn:estimate_linear} and \eqref{eqn:estimate_shape}, the exact dual solutions appearing in the weights must be approximated. To this end, we employ the standard DWR approach based on enriched dual approximations. Let $\widetilde V_h\supset V_h$ and $\widetilde X_h\supset X_h$ denote quadratic finite element spaces defined on the same triangulation $\mathcal T_h$. The enriched dual solutions $\widetilde z_h\in\widetilde V_h$ and $\widetilde\vartheta_h\in\widetilde X_h$ are obtained by solving the corresponding dual problems posed in these enriched spaces. The dual weights are then approximated by
\[
\omega_K^z
\approx
\|\widetilde z_h-I_h\widetilde z_h\|_K
+
h_K^{1/2}\|\widetilde z_h-I_h\widetilde z_h\|_{\partial K},
\]
and analogously for $\omega_K^\vartheta$. Here, $I_h$ denotes the interpolation operator from the enriched spaces onto the corresponding linear finite element spaces. The differences $\widetilde z_h-I_h\widetilde z_h$ and $\widetilde\vartheta_h-I_h\widetilde\vartheta_h$ provide computable approximations of the dual weights appearing in the DWR error estimators.}

\x{Finally, following \cite{BEndtmayer_ULanger_TRichter_ASchafelner_TWick_2024}, we balance the contributions of the different goal functionals by scaling them and define the combined error quantity
\begin{equation}\label{eqn:Q_bound}
    \mathcal Q
    =
    \frac{|\mathcal Q^{c}(u)-\mathcal Q^{c}(u_h)|}
         {|\mathcal Q^{c}(u_h)|}
    +
    \frac{|\mathcal Q^{d}(\theta)-\mathcal Q^{d}(\theta_h)|}
         {|\mathcal Q^{d}(\theta_h)|}.
\end{equation}
In practice, the unknown errors are replaced by the computable estimators $\eta^c$ and $\eta^d$ defined above.}

\subsection{Adaptive sampling and computation of a practical step length}\label{subsec:adaptive_sampling_step_length}

It is well known that for a standard Monte Carlo estimator as in \eqref{eqn:objective_mc_estimator}, the error is of  order $\mathcal{O}(N^{-1/2})$, which means that a sufficiently large sample size $N$ is required to accurately estimate the expectation. However, when constructing objective functionals and gradient-like terms, repeatedly solving the linear system obtained from the finite element discretization can be cumbersome with  large sample sizes. 
Recent advances in machine learning have made stochastic gradient descent (SGD) methods  a promising alternative approach to alleviating the computational burden of robust shape optimization problems;  see, e.g., \cite{SDe_KMaute_ADoostan_2020,LJofre_ADoostan_2022}. The SGD approach produces an unbiased stochastic approximation of the gradient at each optimization step, analogous to the standard Monte Carlo method but using significantly fewer samples. However, in order to control the variance of
the underlying quantity of interest, the sample size must still be chosen adequately. Following the strategies proposed in \cite{RBollapragada_RHByrd_JNocedal_2018,RHByrd_GMChin_JNocedal_2012}, we mitigate this variance by adaptively increasing the number of samples whenever necessary. In this strategy, we begin with a relatively small sample set of size $|S| \ll N$. If the current iteration is likely to yield a meaningful descent direction for the stochastic shape derivative, the sample size is left unchanged. Otherwise, the sample set is enlarged based on an a posteriori error indicator that reflects the variance of the sampled stochastic shape derivative.

\x{After spatial discretization, both the deformation problem and the associated shape derivatives admit finite-dimensional representations. Accordingly, each admissible shape is represented by a vector of design variables, and the resulting optimization problem can be viewed as the minimization of a discrete objective function $\mathcal J_h:\mathcal U_h^{\mathrm{ad}}\subset\mathbb R^n\to\mathbb R$. Let $g_i\in\mathbb R^n$ denote the coefficient vector corresponding to the finite-element discretization of the shape derivative $\diff J^P(\cW,\xi_i)$ associated with the $i$-th sample. We then define the averaged stochastic shape derivative over a sample set $S_k$ by
\[
g_{S_k}
=
\frac{1}{|S_k|}
\sum_{i\in S_k}
g_i.
\]}

Following the methodology of Bollapragada et al.~\cite{RBollapragada_RHByrd_JNocedal_2018}, the a posteriori rule for predicting the size of the next sample set $S_{k+1}$, referred to as the (practical) \emph{augmented inner product test}, is given by
\begin{equation}\label{eq:sample_update}
   |S_{k+1}| = \left\lceil \rho \, |S_k| \right\rceil,
   \qquad \text{with} \qquad
   \rho = \max(\rho^{IT}, \rho^{OT}),
\end{equation}
where
\begin{eqnarray*}
  \rho^{IT}
  &:=&
  \frac{1}{(|S_k|-1)|S_k|\,\nu_{IT}^2\|g_{S_k}\|^4}
  \sum_{i\in S_k}
  \bigl(g_i^Tg_{S_k}-\|g_{S_k}\|^2\bigr)^2, \\
  \rho^{OT}
  &:=&
  \frac{1}{(|S_k|-1)|S_k|\,\nu_{OT}^2\|g_{S_k}\|^2}
  \sum_{i\in S_k}
  \left(
      \|g_i\|^2
      -
      \frac{(g_i^Tg_{S_k})^2}{\|g_{S_k}\|^2}
  \right).
\end{eqnarray*}
Here, $\nu_{IT}$ and $\nu_{OT}$ are positive constants. The update rule \eqref{eq:sample_update} is applied whenever at least one of the conditions
\begin{equation}\label{eqn:sample_condition}
    \rho^{IT} \leq 1
    \qquad\text{and}\qquad
    \rho^{OT} \leq 1
\end{equation}
fails to hold. 


To update the shape within the gradient-based optimization framework, we use the deformation field $\theta_k$ obtained from the deformation problem described in Section~\ref{subsec:descent_direction}. At the continuous level, the shape update can be formally written as
\begin{equation}\label{eqn:update}
\mathcal W_{k+1}
=
(\mathrm{Id}+\alpha_k\theta_k)(\mathcal W_k),
\end{equation}
where $\alpha_k>0$ denotes the optimization step length.

\x{In the numerical implementation, the update is performed on the corresponding discrete design variables. The step length $\alpha_k$ is chosen in accordance with the convergence theory of Bollapragada et al.~\cite[Theorem~3.4]{RBollapragada_RHByrd_JNocedal_2018}. Assuming that the discrete objective function $\mathcal J_h$ satisfies the hypotheses of that theorem, $\alpha_k$ is selected such that
\begin{equation}\label{eqn:step_size}
\alpha_k
\le
\frac{1}{(1+\nu_{IT}^2+\nu_{OT}^2)L},
\end{equation}
where $L$ denotes a Lipschitz constant of the gradient $g_{S_k}$. Under these assumptions, the convergence results of \cite{RBollapragada_RHByrd_JNocedal_2018} ensure that the resulting stochastic gradient descent iteration generates a descent sequence for the discrete optimization problem.}

\x{In practice, the exact value of $L$ is generally unavailable. Therefore, we estimate a local Lipschitz constant adaptively during the optimization process. Since the shapes are evolved through the Hamilton--Jacobi equation, we approximate the distance between two consecutive shapes using the displacement induced by the deformation field. Following \cite{DLuft_KWelker_2019}, we define the geometric distance proxy
\[
    \widehat d(\cW_k,\cW_{k-1})
    =
    \int_{\cW_k}
    \min_{y\in\cW_{k-1}}
    \|x-y\|\,dx .
\]}

\noindent Recall that points are evolved according to
\[
    \dot x(t)=\theta(x(t)),
    \qquad
    x(0)=x_0.
\]
For a fictitious time interval $t>0$, corresponding to the final time of the discrete Hamilton--Jacobi evolution \eqref{eqn:HJE_discrete}, we have
\begin{equation}\label{eqn:x_approx}
    x(t)
    =
    x(0)
    +
    \int_0^t \theta(x(\tau))\,d\tau
    \approx
    x(0)+t\,\theta(x(0)).
\end{equation}
Thus,
\begin{align}
\widehat d(\cW_k,\cW_{k-1})
&=
\int_{\cW_k}
\min_{x_{k-1}\in\cW_{k-1}}
\|x-x_{k-1}\|\,dx
\nonumber\\
&\lesssim
t\,\|\theta_{k-1}\|_{\infty}\,|\cW_k|,
\label{eqn:distance_approx}
\end{align}
where $\|\theta_{k-1}\|_\infty$ denotes the maximum norm of the deformation field computed from the sampled stochastic gradient $g_{S_{k-1}}$.

We then estimate a local shape-dependent Lipschitz constant by the
secant-type approximation
\begin{equation}\label{eqn:lipschitz_estimated}
L_k
\approx
\frac{\|g_{S_k}-g_{S_{k-1}}\|}
     {\widehat d(\cW_k,\cW_{k-1})}
\approx
\frac{
\|g_{S_k}-g_{S_{k-1}}\|
}{
t\,\|\theta_{k-1}\|_{\infty}\,|\cW_k|
}.
\end{equation}
Finally, the adaptive step length is obtained by replacing $L$ with $L_k$ in \eqref{eqn:step_size}.


\section{Adaptive procedure} \label{sec:algorithm} 

In this section, we present an adaptive procedure for solving the robust shape optimization problem \eqref{eqn:objective}; see Algorithm~\ref{alg:full_procedure}. We also provide detailed descriptions of the algorithm's main subroutines. Recall that $k \geq 0$ denotes the optimization iteration, and let $\cW_k$ be the current shape with deformable boundary $\partial\cW_k$. The associated finite element spaces are denoted by
\[
V_h^k := V_h(\mathcal T_h^k),
\qquad
X_h^k := X_h(\mathcal T_h^k),
\]
where $\mathcal T_h^k$ is the computational mesh on $\cD$ used to represent the current shape $\cW_k$.

\begin{algorithm}[htp!]
    \begin{algorithmic}[1]
        \Require Working domain $\cD$, initial level-set function $\psi^0$, parameters $\alpha_0,\nu_{OT},\nu_{IT},\Lambda,\varsigma,\epsilon,\tau_1,\tau_2>0$,
        initial sample size $|S_0|$, KL eigenpairs $\{(\lambda_i,b_i)\}_{i=1}^{M}$,
        initial triangulation $\mathcal T_h^0$ with mesh size $h^0$,
        reference mesh size $h^*$, and tolerance $\texttt{tol}_\eta$.
        \State $\cW_0 \leftarrow \{x\in\cD:\psi^0(x)<0\}$ 
        \State $k \leftarrow 0$ 
        \State $\alpha_k \leftarrow \alpha_0$
        \State $\{\xi^i\}_{i=1}^{M} \leftarrow \texttt{compute\_kle}(\lambda_i,b_i)$
        \While{stopping criteria are not satisfied}
            \For{$i=1, \ldots, |S_k|$}
                \State $u_h^i,z_h^i,\theta_h^i,\vartheta_h^i \leftarrow \texttt{solve\_weak\_forms} (\mathcal T_h^k,\cW_k,\xi^i)$
                \State $\eta_K^{c,i},\eta_K^{d,i} \leftarrow \texttt{estimate\_dwr} (\mathcal T_h^k,\cW_k,u_h^i,z_h^i,\theta_h^i,\vartheta_h^i,\xi^i)$
                \State $(J_k^P)^i \leftarrow \texttt{compute\_cost} (\mathcal T_h^k,\cW_k,u_h^i,\xi^i)$
                \State $(\diff J_k^P)^i \leftarrow \texttt{compute\_shape\_derivative} (\mathcal T_h^k,\cW_k,u_h^i,z_h^i,\theta_h^i,\xi^i)$
            \EndFor
            \State Compute the sample averages $\eta_k^{S_k}$, $(J_k^P)^{S_k}$, $(\diff J_k^P)^{S_k}$, and $\theta_h^{S_k}$.
            \If{$\eta_k^{S_k} > \mathrm{tol}_{\eta}$ \textbf{and} $h_{\min}(\mathcal T_h^k)>h^*$}
                \State $\mathcal T_h^k \leftarrow \texttt{mark\_and\_refine} (\mathcal T_h^k,\eta_k^{S_k})$
                \State \textbf{continue} (Restart the current optimization iteration on the refined mesh).
            \Else
                \State $|S_{k+1}| \leftarrow \texttt{estimate\_sample\_size} \bigl( (\diff J_k^P)^{S_k}, \{(\diff J_k^P)^i\}_{i\in S_k}, \nu_{IT}, \nu_{OT} \bigr)$
                \If{$k > 0$} 
                    \State $\alpha_k \leftarrow \texttt{estimate\_step\_length} \bigl( (\diff J_k^P)^{S_k}, (\diff J_{k-1}^P)^{S_{k-1}}, \cW_k \bigr)$ 
                \EndIf
                \State $\mathcal T_h^{k+1} \leftarrow \texttt{reset\_mesh} (\mathcal T_h^k)$
                \State $\psi^{k+1} \leftarrow \texttt{solve\_hje} (\psi^k,\theta_h^{S_k},\alpha_k)$
                \State $\cW_{k+1} \leftarrow \{x\in\cD:\psi^{k+1}(x)<0\}$
                \State $k \leftarrow k+1$
            \EndIf
\EndWhile
\caption{An adaptive framework for robust shape optimization.}
\label{alg:full_procedure}
\end{algorithmic}
\end{algorithm}

At the beginning of the procedure, the KL expansion \eqref{eqn:KLexpansion} is generated in the subroutine \texttt{compute\_kle} using the prescribed eigenpairs $\{(\lambda_i,b_i)\}_{i=1}^{M}$. In the subroutine \texttt{solve\_weak\_forms}, we then solve the primal problems \eqref{eqn:weak_form_continuous} and \eqref{eqn:weak_form_theta}, along with their corresponding dual problems, by employing the finite element spaces $V_h^k$ and $X_h^k$ associated with the current shape $\cW_k$. Next, the dual-weighted residual estimators given in \eqref{eqn:estimate_linear} and \eqref{eqn:estimate_shape} are computed in the subroutine \texttt{estimate\_dwr}. During the computation of the corresponding error indicators, it is additionally necessary to solve the enriched dual problems
\[
a^*(\widetilde z_h, v)
=
\langle \delta_u \mathcal Q^c(u_h), v \rangle
\qquad
\forall v\in\widetilde V_h^k
\]
and
\[
b^*(\widetilde\vartheta_h,\varphi)
=
\langle \delta_\theta \mathcal Q^d(\theta_h),\varphi\rangle
\qquad
\forall \varphi\in\widetilde X_h^k,
\]
where $\widetilde V_h^k\supset V_h^k$ and $\widetilde X_h^k\supset X_h^k$ denote quadratic finite element spaces defined on the same triangulation. The solutions $\widetilde z_h$ and $\widetilde\vartheta_h$ are then employed to approximate the dual weights appearing in the DWR error estimators. Subsequently, the penalized compliance cost \eqref{eqn:objective_mc_estimator} and the corresponding shape derivative \eqref{eqn:derivative_mc_estimator} are evaluated for each sample realization $\xi^i$ in the subroutines \texttt{compute\_cost} and \texttt{compute\_shape\_derivative}, respectively. The subroutine \texttt{mark\_and\_refine} selects elements of $\mathcal{T}_h^k$ to be refined using Dörfler's bulk-marking strategy \cite{WDorfler_1996}, and then refines the marked elements by bisection while preserving the conformity of $\mathcal{T}_h^k$. This refinement is performed whenever the scaled error indicator
$\eta_k^{S_k}$ exceeds the prescribed tolerance
$\mathrm{tol}_{\eta}$ while the minimum mesh size of the current triangulation remains larger than the prescribed reference mesh size $h^*$.

Once the discretization error has been assessed, we apply the augmented inner-product criterion described in Section~\ref{subsec:adaptive_sampling_step_length} to determine whether the sample size should be increased in the subroutine \texttt{estimate\_sample\_size}. Subsequently, the optimization step length $\alpha_k$ is computed in the subroutine \texttt{estimate\_step\_length} using the bound \eqref{eqn:step_size}, where the Lipschitz constant is approximated by the local estimate $L_k$ defined in \eqref{eqn:lipschitz_estimated}.

Before updating the current shape, the mesh is reset to the initial triangulation. This prevents the accumulation of locally refined meshes during the optimization process and simplifies the subsequent level-set evolution. The Hamilton--Jacobi equation \eqref{eqn:hje} is then solved by the finite difference scheme described in Section~\ref{subsec:descent_direction} within the subroutine \texttt{solve\_hje}. The updated level-set function $\psi^{k+1}$ determines the new shape
\[
\cW_{k+1}
=
\{x\in\cD:\psi^{k+1}(x)<0\}.
\]
The optimization loop is terminated whenever one of the following
conditions is satisfied:
\begin{itemize}
    \item Setting $J_{S_k}^P := (J_k^P)^{S_k}$, the relative change of the objective value over the previous five
    iterations satisfies
    \[
    \max_{1\le i\le \min(5,k)}
    \frac{\left|J_{S_k}^P-J_{S_{k-i}}^P\right|}
         {|J_{S_k}^P|}
    \le 10^{-2},
    \]
    and the volume constraint is satisfied up to the prescribed tolerance,
    \[
    \left|
    \frac{|\cW_k|}{|\cD|}
    -
    \varsigma
    \right|
    \le 5\times 10^{-3},
    \]
    \item or a prescribed maximum number of optimization iterations has
    been reached.
\end{itemize}


\begin{remark}
Mesh coarsening is as important as mesh refinement in level-set--based shape optimization because it enables the computational mesh to adapt to the evolving geometry by eliminating unnecessarily fine elements while maintaining the accuracy of the numerical approximation. Such a mechanism is particularly beneficial when regions that previously required local refinement no longer contain relevant geometric features.

Since mesh coarsening is not currently supported in the FEniCS/DOLFIN framework, we employ the alternative strategy used in Algorithm~\ref{alg:full_procedure}, namely resetting the mesh to the initial triangulation after each optimization step. Although this approach is sufficient for the benchmark problems considered in this work, the incorporation of a genuine adaptive coarsening procedure would be desirable for large-scale and more complex applications, where long optimization runs may otherwise lead to unnecessarily fine meshes and increased computational costs.
\end{remark}


\section{Numerical experiments} \label{sec:numeric}

\x{In this section, we investigate the performance of the proposed robust shape optimization framework through numerical experiments inspired by aerospace landing systems \cite{SAKhan_ZMehmood_ZAfshan_2021,JWong_LRyan_IYKim_2018} and legged-robot locomotion \cite{AGaathon_ADegani_2022}. In aerospace landing structures, touchdown events may generate impact loads with uncertain magnitude and direction due to variations in landing conditions and terrain properties. Such uncertainties can significantly affect the load-transfer mechanism within the structure and often lead to optimal designs that differ substantially from those obtained under deterministic loading assumptions. Similarly, in legged-robot locomotion, leg components experience repeated and intermittent contact events with the ground, resulting in uncertain normal and tangential forces. Variations in terrain geometry, surface slopes, and contact conditions may substantially alter the structural response. In both settings, the objective is to design lightweight leg- or strut-like structures that remain mechanically efficient under stochastic loading conditions.}

\x{The purpose of these numerical investigations is not to model specific engineering systems in full detail, but rather to evaluate the effectiveness of the proposed adaptive optimization framework in representative uncertainty-driven scenarios. To this end, we focus on compliance minimization under random loads and compare three settings: (i) a randomness study on a fixed mesh, (ii) an adaptive mesh refinement study using full Monte Carlo sampling, and (iii) the proposed fully adaptive framework combining adaptive sampling and adaptive mesh refinement.}

In the finite element approximation of the solutions, continuous piecewise-linear Lagrange elements are employed on the crossed triangulation $\mathcal{T}_h$, resulting in $2(n_x+1)(n_y+1) + 2n_x n_y$ degrees of freedom (DoF). Random input fields are represented by a truncated Gaussian Karhunen--Loève (KL) expansion with a separable exponential covariance kernel
\begin{equation*}
    \textnormal{Cov}[\varkappa](x, \widetilde{x}) 
    = \kappa_{\varkappa}^2 
      \prod_{j=1}^2 e^{-|x_j - \widetilde{x}_j| / \ell_j},
\end{equation*}
where $\ell_j$ denotes the correlation length in the $j$-th spatial direction and $\kappa_\varkappa$ is the standard deviation
on the computational domain $\cD = [0, l_x] \times [0, l_y]$.  Unless otherwise stated, the parameters used in the simulations are provided in Table~\ref{tab:para}. All computations are carried out on an Intel(R) Core(TM) i5-6500 @ 3.20~GHz processor with 8~GB RAM, using the FEniCS/DOLFIN library and based on the implementation developed by Laurain \cite{ALaurain_2018} \x{together with custom routines for adaptive mesh refinement, stochastic sampling, and level-set evolution. The source code used to produce the numerical results reported in this work will be made publicly available upon acceptance of the manuscript, thereby ensuring the reproducibility of the presented findings.}

\begin{table}[htp!]
\caption{Descriptions of the parameters used in the numerical simulations.}
\label{tab:para}
\centering
\begin{tabular}{lll}
\hline
Parameter & Description & Value \\
\hline
$\epsilon$
& weak material parameter in \eqref{eqn:hook_ersatz}
& $10^{-3}$ \\

$\varsigma$
& prescribed volume fraction in \eqref{eqn:admissible}
& $0.3$ \\

$\Lambda$
& volume-constraint penalty parameter in \eqref{eqn:optimization_penalized_fixed}
& $500$ \\

$\nu_{IT}$
& inner-product test parameter in \eqref{eq:sample_update}
& $0.6$ \\

$\nu_{OT}$
& orthogonality test parameter in \eqref{eq:sample_update}
& $5.8 \approx \tan(80^\circ)$ \\

$\tau_1,\tau_2$
& deformation bilinear-form parameters in \eqref{eqn:weak_form_theta}
& $10^{3},\,1$ \\

$|S_0|$
& initial Monte Carlo sample size
& $2$ \\

$\alpha_0$
& initial optimization step length
& $0.01$ \\

$\ell_1, \ell_2$
& correlation length in the KL covariance kernel
& $1$ \\

$\texttt{tol}_{\eta}$
& tolerance for the DWR error estimator
& $0.1$ \\

$h^*$
& minimum admissible mesh size
& $1/360$ \\
\hline
\end{tabular}
\end{table}

\begin{figure}[htp!]
    \centering
    \includegraphics[width=0.32\linewidth]{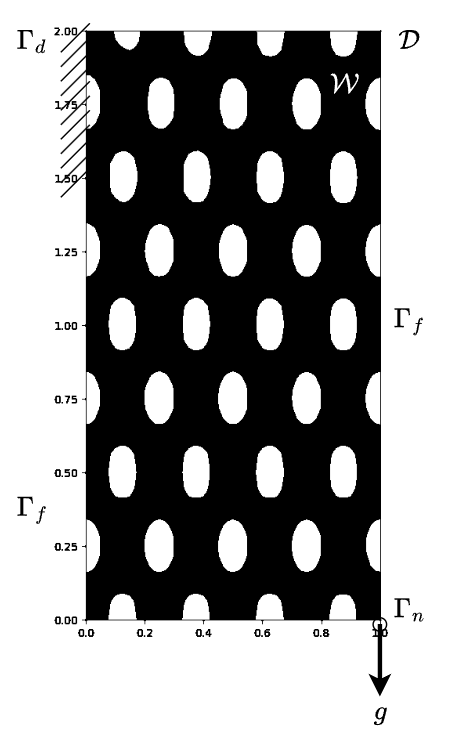}
    \caption{Initial level-set representation of the design domain on the computational domain $\mathcal{D}$, together with the prescribed boundary conditions and the external load $g$ applied at $(0,1)$.}
    \label{fig:leg_design_setup}
\end{figure}

As a benchmark example, we consider the compliance minimization of carrier legs in aerospace and robotic structures. In this setting, the isotropic material parameters, including the dummy material approximation, are  given by
\[
    \mu = \frac{E}{2(1+\nu)}
    \qquad \text{and} \qquad
    \lambda = \frac{E\nu}{(1+\nu)(1-2\nu)},
\]
where the Young's modulus is set to $E = 1$ and the Poisson ratio to $\nu = 0.3$. We assume that there is no body force, that is, $f=0$. The initial mesh is generated using $n_x = 60$ and $n_y = 120$, resulting in $2 \times 14{,}581$ degrees of freedom (DoF). \x{Unless otherwise stated, 30 samples are used in the full sampling approach.} With correlation lengths chosen as $\ell_1 = \ell_2 = 1$, we compute the first $M$ eigenpairs $\{(\lambda_i, b_i)\}_{i=1}^M$ of the KL expansion to capture 90\% of the total energy, with an upper bound of $M_{\max} = 100$. Further, the random loading (i.e., contact force) is prescribed as 
\[
    g = (10\cos\phi,\, 10\sin\phi),
\]
where the random loading angle is realized through
\begin{equation*}
    \phi = \overline{\phi} + \kappa_\phi \sum_{i=1}^M \sqrt{\lambda_i}\, b_i(x)\, \xi_i,
\end{equation*}
with mean angle $\overline{\phi} = 90^\circ$ and standard deviation $\kappa_\phi \in \{5^\circ, 10^\circ, 30^\circ\}$. Here, $\xi = (\xi_1,\dots,\xi_M)$ denotes an i.i.d.\ vector of Gaussian random variables satisfying $\xi_i \sim \mathcal{N}(0,1)$. The load $g(x,\omega)$ is applied at the point $(1,0)$, i.e., $\Gamma_n = \{1\} \times \{0\}$. The Dirichlet boundary is specified as $\Gamma_d = \{0\} \times [1.5,2]$, while the remaining part of the boundary is taken to be the free boundary $\Gamma_f$. Further, the initial domain $\mathcal{W}_0$ is constructed using the initial  level-set function
$ \psi^0(x,y) =  -\cos(8\pi x)\cos(4\pi y) - 0.5$ as illustrated in Figure~\ref{fig:leg_design_setup}.

To compare the computational effort required by the different numerical experiments, we introduce the computational cost indicator
\begin{equation}\label{index}
    \texttt{CI}
    \propto
    \sum_{i=1}^{K}
    \sum_{j=1}^{Q}
    N_i\,\varrho(d_{ij}),
\end{equation}
where $K$ denotes the total number of optimization iterations, $Q$ is the number of mesh refinement levels, $d_{ij}$ is the number of degrees of freedom at the $i$-th optimization iteration and the $j$-th refinement level, and $N_i$ is the number of Monte Carlo samples used at the $i$-th optimization iteration. The function $\varrho$ represents the computational complexity of the linear solver as a function of the number of degrees of freedom. As the main linear solver, we employ the MUMPS implementation available in FEniCS/DOLFIN, whose complexity for two-dimensional problems is often approximated by
$\mathcal O(d^{3/2})$; see, e.g., \cite{PRAmestoy_ISDuff_JVLexcellent_JKoster_2001}. 

\subsection{Randomness study}

We first investigate the effect of uncertainty in the loading angle for various values of the standard deviation $\kappa_\phi$. These values represent scenarios in which the landing approach is maintained despite substantial deviations from the nominal touchdown angle or when leg impacts occur on terrains with varying levels of roughness. Accordingly, Figure~\ref{fig:random_study_angle} presents the optimized designs obtained for $\kappa_\phi \in \{0^\circ, 5^\circ,10^\circ,30^\circ\}$. \x{As $\kappa_\phi$ increases, the optimized designs exhibit thicker primary load paths, indicating enhanced robustness. The minor changes observed between the $\kappa_\phi=10^\circ$ and $30^\circ$ cases suggest convergence toward a robust topology under high uncertainty.}

\begin{figure}[h!]
         \centering
         \begin{subfigure}[b]{0.24\textwidth}
             \centering
             \includegraphics[width=\textwidth]{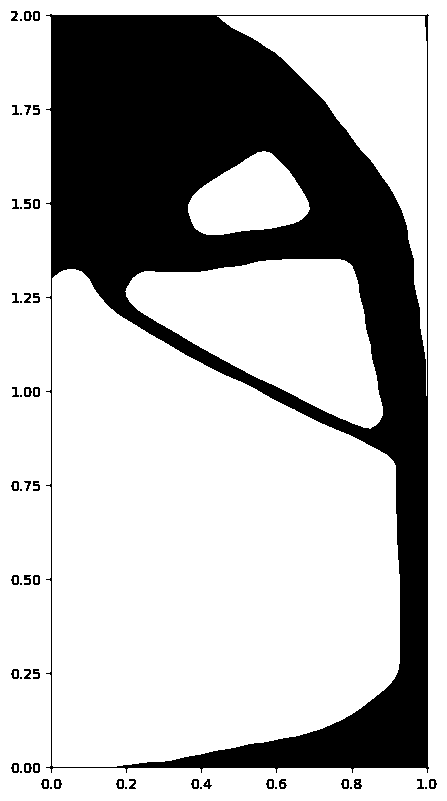}
             \caption{$\kappa_\phi = 0^\circ$.}
         \end{subfigure}
         \hfill         
         \begin{subfigure}[b]{0.23\textwidth}
             \centering
             \includegraphics[width=\textwidth]{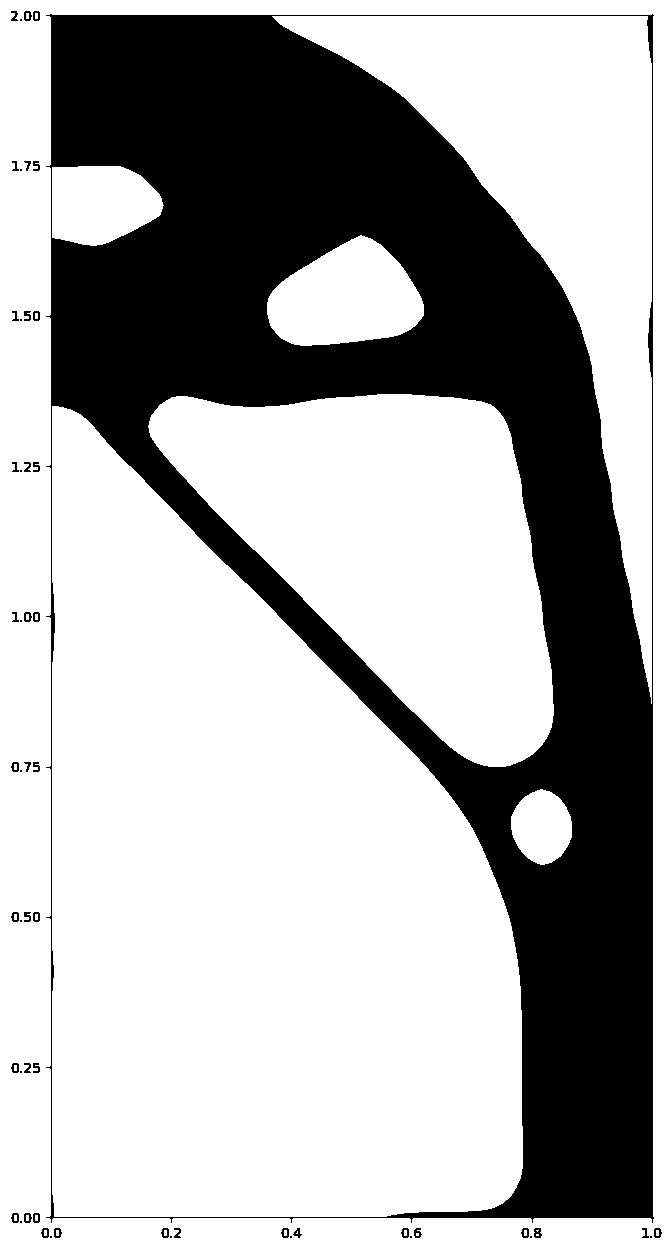}
             \caption{$\kappa_\phi = 5^\circ$.}
         \end{subfigure}
         \hfill
         \begin{subfigure}[b]{0.23\textwidth}
             \centering
             \includegraphics[width=\textwidth]{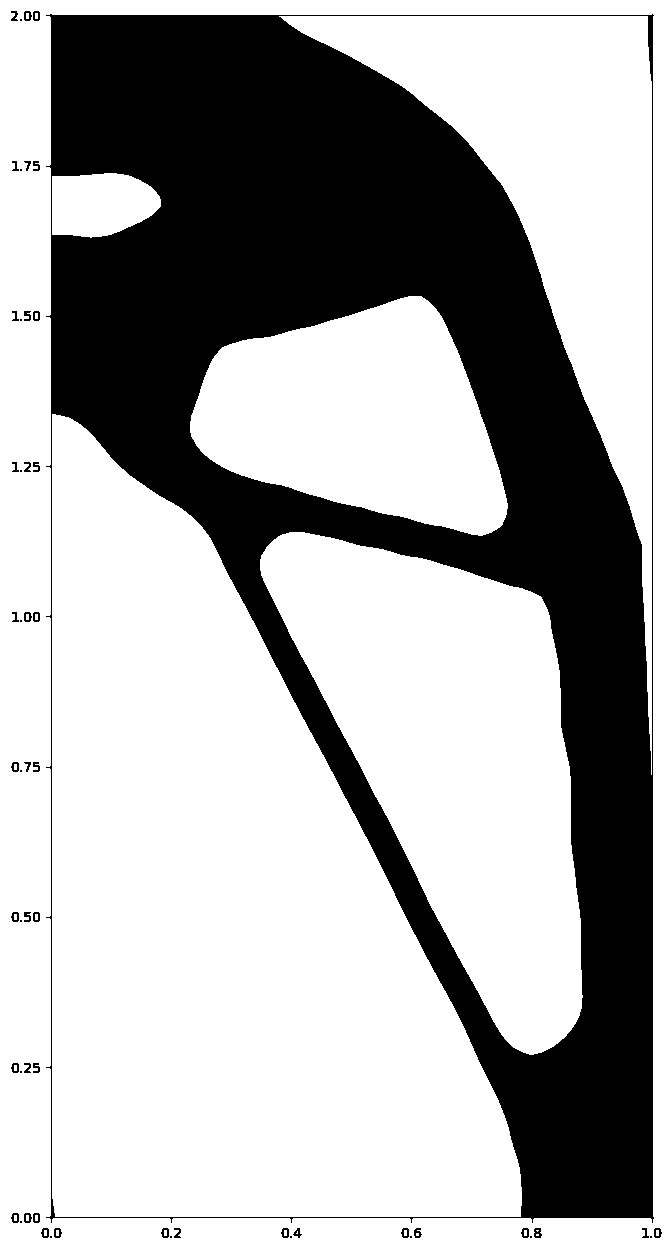}
             \caption{$\kappa_\phi = 10^\circ$.}
         \end{subfigure}
         \hfill
         \begin{subfigure}[b]{0.24\textwidth}
             \centering
             \includegraphics[width=\textwidth]{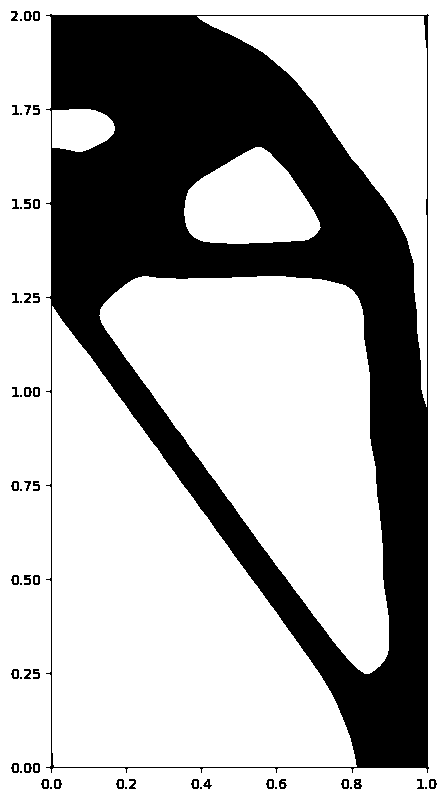}
             \caption{$\kappa_\phi = 30^\circ$.}
         \end{subfigure}
            \caption{Final optimized designs obtained for various loading-angle deviations $\kappa_\phi$ using full Monte Carlo sampling on a fixed mesh.}
            \label{fig:random_study_angle}
\end{figure}

Next, we investigate the effect of randomness on a fixed mesh and demonstrate the effectiveness of the adaptive sampling strategy introduced in Section~\ref{subsec:approx_parametric}. Figure~\ref{fig:random_study_final_design} shows the final designs obtained for a loading-angle standard deviation of $\kappa_\phi = 30^\circ$ using full and adaptive sampling. It can be observed that the optimized designs produced by the two approaches have nearly identical topology. However, the computational effort is significantly reduced when adaptive sampling is employed; see Table~\ref{table:random_study_comp_index}. According to the computational cost indicator $\texttt{CI}$ reported in Table~\ref{table:random_study_comp_index}, gradually increasing the sample size, rather than using the full sample set at every optimization iteration, substantially reduces the computational cost while still providing accurate approximations of the mean compliance obtained from full Monte Carlo sampling; see Figure~\ref{fig:randomness_study_mean_sample}.

\begin{figure}[h!]
         \centering        
         \begin{subfigure}[b]{0.25\textwidth}
             \centering
             \includegraphics[width=\textwidth]{fs_umc_design_rev.png}
             \caption{Full sampling.}
         \end{subfigure}
         \qquad
         \begin{subfigure}[b]{0.25\textwidth}
             \centering
             \includegraphics[width=\textwidth]{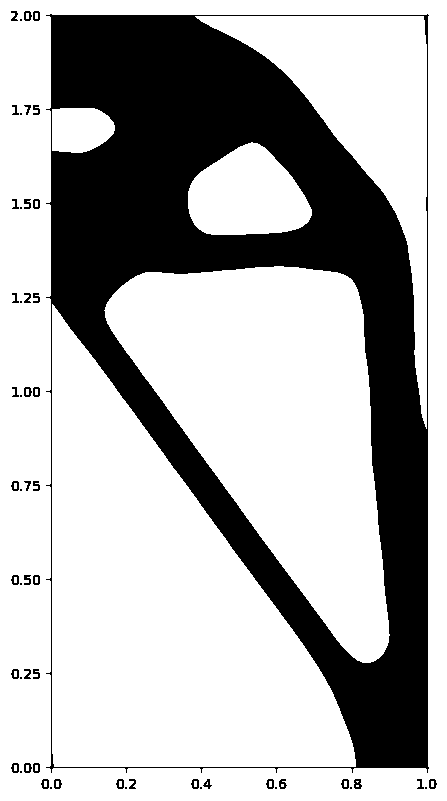}
             \caption{Adaptive sampling.}
         \end{subfigure}
         \caption{Final optimized designs on a fixed mesh using full and adaptive sampling approaches for a loading-angle deviation of $\kappa_\phi = 30^\circ$.}
         \label{fig:random_study_final_design}
\end{figure}

\begin{table}[htp!]
    \caption{Comparison of the computational effort of the full and adaptive sampling approaches in terms of the number of samples, total runtime, and the computational cost indicator defined in \eqref{index}.}\label{table:random_study_comp_index}
	\centering
\begin{tabular}{l|l|l|l|l}
	                    & DoF        & $N$                & $T_{\text{total}}$[hours] & \texttt{CI} \\ \hline
    full sampling    & $29,\!162$ & $30$               & $4.94$              & $44.7 \cdot 10^9$ \\
    adaptive sampling   & $29,\!162$ & $2 \rightarrow 6$  & $1.18$              & $8.54 \cdot 10^9$
    \end{tabular}
\end{table}

\begin{figure}[h!]
        \centering
        \begin{subfigure}[b]{0.48\textwidth}
            \centering
            \includegraphics[width=\textwidth]{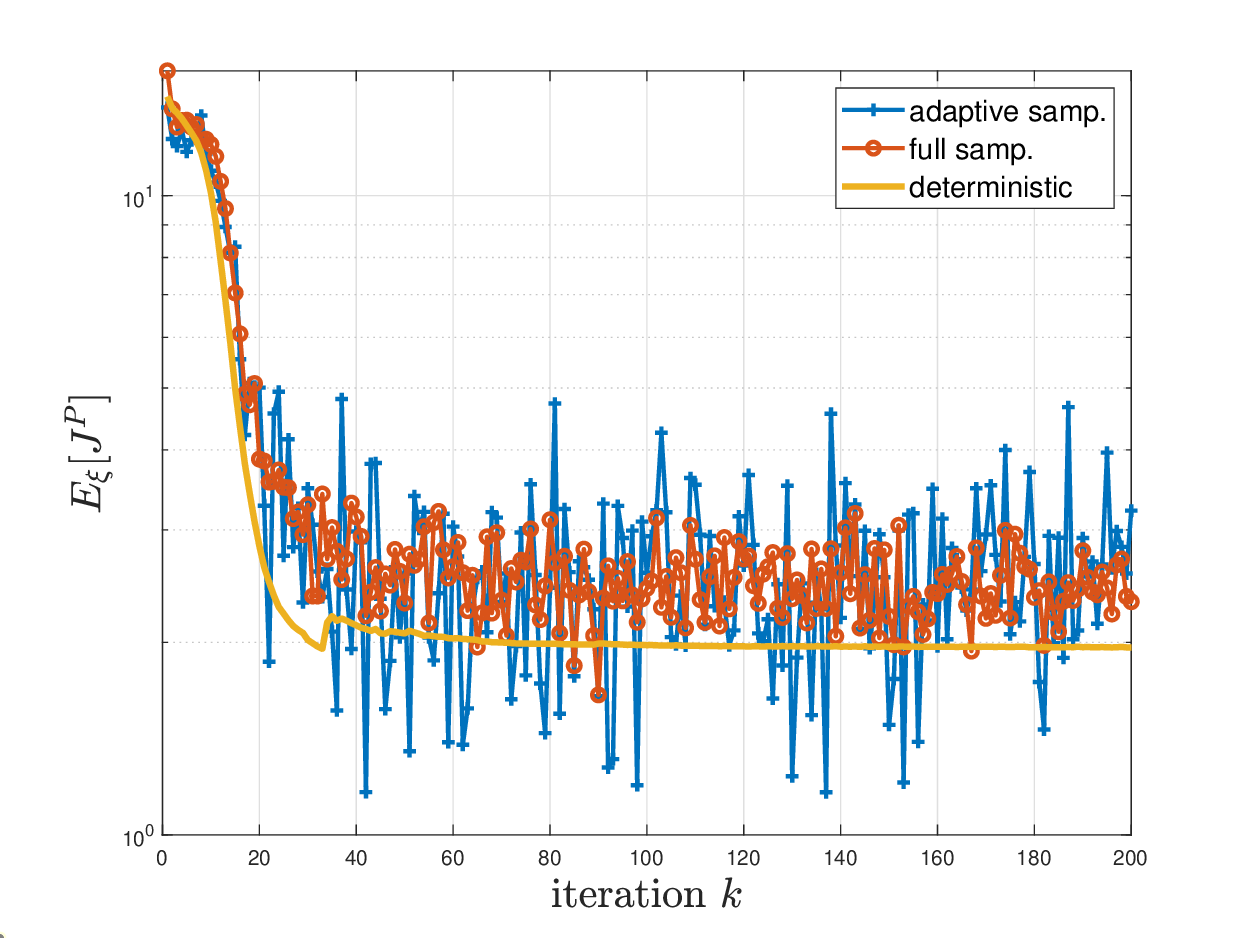}
            \caption{Mean compliance.}
        \end{subfigure}
        \begin{subfigure}[b]{0.48\textwidth}
            \centering
            \includegraphics[width=\textwidth]{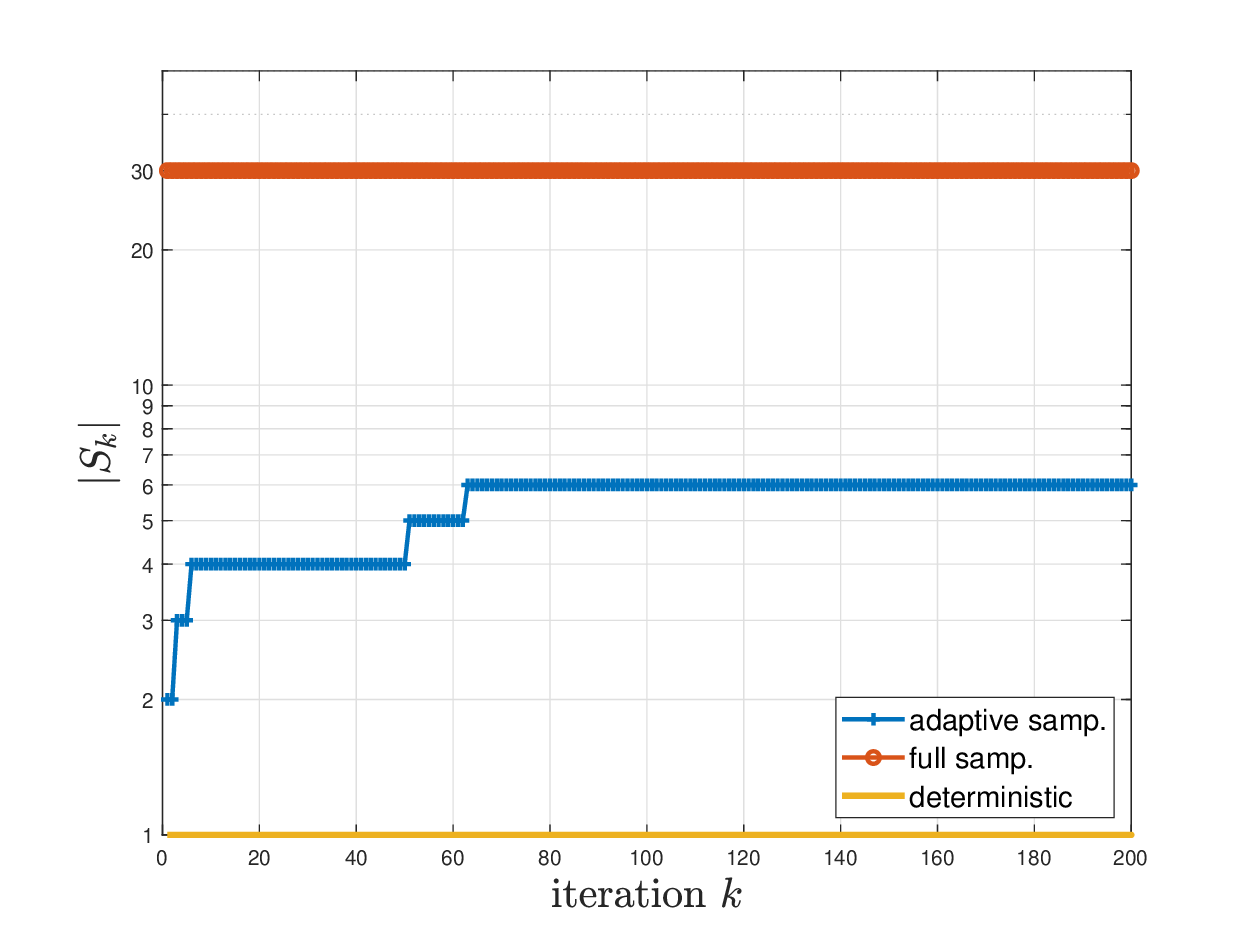}
            \caption{Number of Monte Carlo samples.}
        \end{subfigure}
        \caption{Evolution of the mean compliance $\mathbb{E}_{\xi}[J^P]$ and the size of the Monte Carlo sample set $|S_k|$ for deterministic and random loading angles ($\kappa_\phi = 30^\circ$) on a fixed mesh.}
            \label{fig:randomness_study_mean_sample}
\end{figure}

\subsection{Mesh refinement study}

Next, we investigate the effect of mesh refinement under the full-sampling setting. Three mesh configurations are considered: a fixed coarse mesh with $29{,}162$ DoF, a fixed fine mesh with $260{,}282$ DoF, and an adaptively generated mesh whose size varies from $29{,}162$ to $43{,}288$ DoF during the optimization process. Table~\ref{table:mesh_study_comp_index} compares the computational effort associated with these mesh configurations. \x{The results indicate that although the computational cost is substantially reduced in terms of the computational index \texttt{CI}, the total computation time remains comparable to that of the fine-mesh approach. This is primarily due to the additional time required for repeatedly solving the dual problems during the adaptive mesh refinement process.} The evolution of the compliance functional is shown in Figure~\ref{fig:mesh_study_compliance}, while the final values of the optimization quantities are reported in Table~\ref{table:mesh_study_final values}. The corresponding optimized designs are displayed in Figure~\ref{fig:mesh_study_final_design}. \x{Figure~\ref{fig:mesh_study_final_design} shows that the adaptive and fine meshes lead to very similar final topologies. Furthermore, Figure~\ref{fig:mesh_study_compliance} shows that the adaptive mesh closely reproduces the fine-mesh response and significantly improves upon the coarse-mesh solution. This is confirmed by the results in Table~\ref{table:mesh_study_final values}, where the adaptive mesh yields optimization quantities much closer to those of the fine mesh, demonstrating an effective balance between accuracy and computational effort.}

 \begin{table}[htp!]
	\caption{Comparison of the computational effort between the fixed coarse mesh, fixed fine mesh, and adaptively generated mesh under full sampling in terms of the degrees of freedom, computational runtime, and the computational index defined in~\eqref{index}.}\label{table:mesh_study_comp_index}
	\centering
\begin{tabular}{l|l|l|l|l}
	               & DoF                              & $N$  & $T_{\text{total}}$[hours]  & \texttt{CI} \\ \hline
    fixed coarse mesh   & $29,\!162$                       & $30$ & $4.94$              & $4.47 \cdot 10^{10}$ \\
    fixed fine mesh     & $260,\!282$                      & $30$ & $45.6$              & $119 \cdot 10^{10}$ \\
    adaptive mesh & $29,\!162 \rightarrow 43,\!288$  & $30$ & $41.5$              & $12.5 \cdot 10^{10}$
    \end{tabular}
\end{table}
 
\begin{figure}[htp!]
    \centering
    \includegraphics[width=0.60\textwidth]{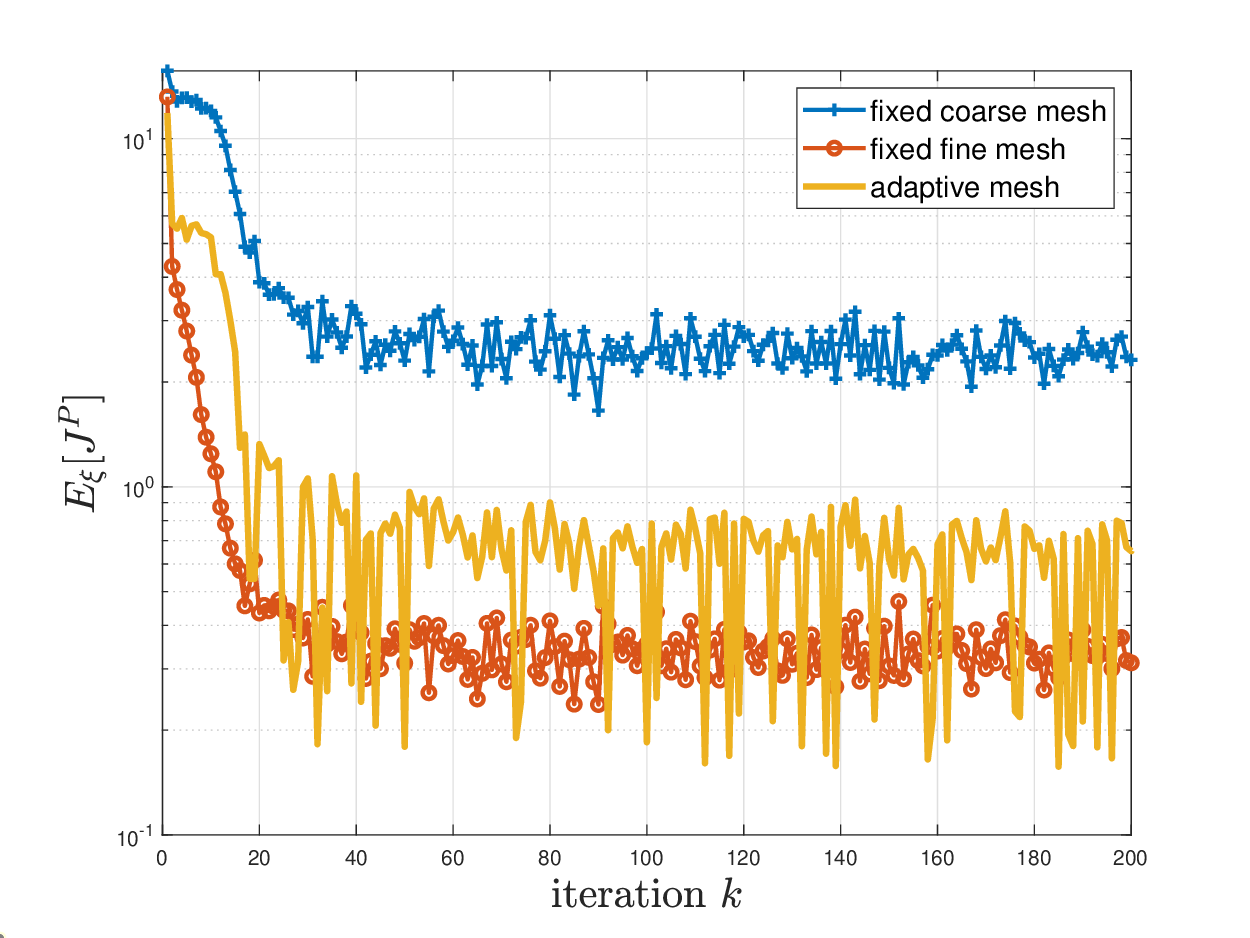}
    \caption{Evolution of the mean compliance $\mathbb{E}_{\xi}[J^P]$ obtained using the fixed coarse mesh, the fixed fine mesh, and the adaptively generated mesh under full Monte Carlo sampling.}
    \label{fig:mesh_study_compliance}
\end{figure}

\begin{table}[htp!]
    \caption{Comparison of the final values of the optimization quantities for the mesh configurations considered in the refinement study.}\label{table:mesh_study_final values}
	\centering
\begin{tabular}{llll}
	Quantity & Adaptive & Coarse Mesh & Fine Mesh \\ \hline
    $\E_{\xi}[J^P]$ & 0.65 & 2.31 & 0.31 \\
    $\cQ$ in \eqref{eqn:Q_bound} & 0.05 & 0.112 & 0.002 \\
    $\eta^c$  & 0.06 & 0.18 & 0.02 \\
    $\eta^d$  & 0.16 & 0.39 & 0.04 \\
    $\alpha_k$ & 0.01 & 0.017 & 0.03 \\
    $\E_{\xi}[\diff J^P]$ & 12.3 & 25.3 & 8.80 \\
    $\mathbb{V}_{\xi}[\diff J^P]$ & 913.9 & 3560.9 & 486.9
    \end{tabular}
\end{table}

 \begin{figure}[htp!]
         \centering
         \begin{subfigure}[b]{0.24\textwidth}
             \centering
             \includegraphics[width=\textwidth]{fs_umc_design_rev.png}
             \caption{Coarse design.}
         \end{subfigure}
         \hfill         
         \begin{subfigure}[b]{0.24\textwidth}
             \centering
             \includegraphics[width=\textwidth]{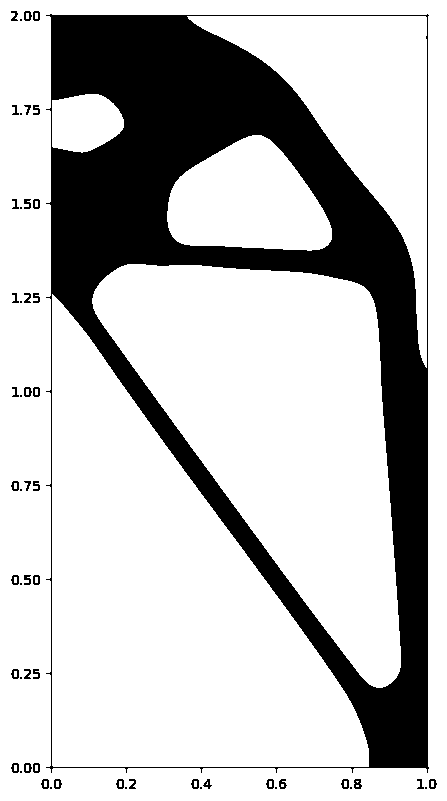}
             \caption{Fine design.}
         \end{subfigure}
         \hfill
         \begin{subfigure}[b]{0.23\textwidth}
             \centering
             \includegraphics[width=\textwidth]{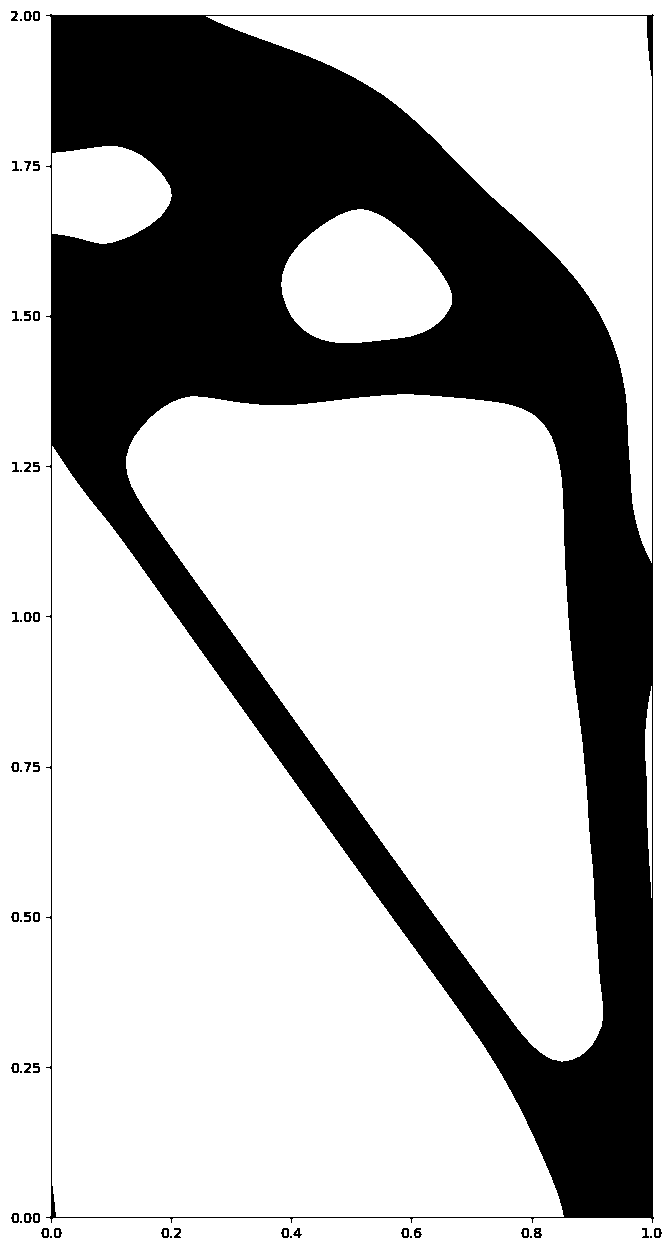}
             \caption{Adaptive design.}
         \end{subfigure}
         \hfill
         \begin{subfigure}[b]{0.23\textwidth}
             \centering
             \includegraphics[width=\textwidth]{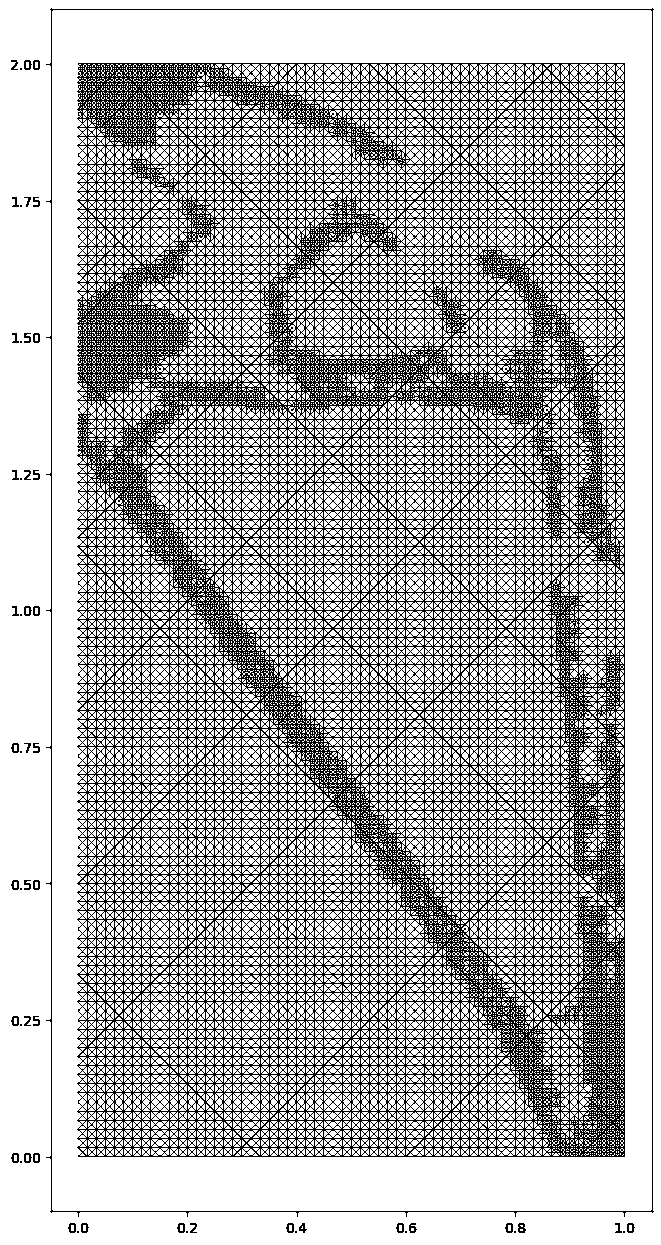}
             \caption{Adaptive mesh.}
         \end{subfigure}
         \caption{Final optimized designs obtained on the fixed coarse mesh with $29{,}162$ degrees of freedom (a), the fixed fine mesh with $260{,}282$ degrees of freedom (b), and the adaptively generated mesh with $43{,}288$ degrees of freedom (c). The corresponding adaptive mesh is shown in (d).}
         \label{fig:mesh_study_final_design}
\end{figure}

\subsection{Full adaptivity study}

Last, we employ both adaptive sampling and adaptive mesh refinement. Compared with the fixed fine mesh under full Monte Carlo sampling, the fully adaptive framework accurately captures the stochastic response while requiring substantially fewer samples and degrees of freedom (DoF); see Figure~\ref{fig:combined_evolution} and Table~\ref{table:combined_final_values}. Moreover, the adaptive procedure simultaneously balances the discretization error and the computational effort associated with the spatial and stochastic approximations. In the numerical simulations, the total runtime for the fixed fine mesh with $260{,}282$ DoF under full sampling is $45.6$ hours, whereas the fully adaptive framework, employing meshes ranging from $29{,}162$ to $43{,}378$ DoF and Monte Carlo sample sizes between $2$ and $6$, requires only $8.72$ hours; see Table~\ref{table:combined_comp_index}. According to the computational cost indicator \texttt{CI}, the fully adaptive strategy reduces the computational effort by nearly two orders of magnitude.

\x{The resulting optimized designs are shown in Figure~\ref{fig:combined_final_design}. Despite the substantial reduction in both the number of degrees of freedom and Monte Carlo samples, the adaptive design preserves the main structural features of the reference fine-mesh solution. Furthermore, Figure~\ref{fig:combined_evolution}(a) shows that the adaptive framework follows a convergence trend comparable to that of the fixed strategy. The larger oscillations observed in the adaptive compliance history are likely due to the increased variance of the stochastic gradient estimates associated with the reduced Monte Carlo sample sizes, while the mesh adaptation process may further contribute to these fluctuations. Nevertheless, the final optimization quantities reported in Table~\ref{table:combined_final_values} remain comparable to those obtained with the fixed fine mesh, demonstrating that the combined adaptive framework provides an effective balance between accuracy and computational efficiency.}

\begin{figure}[htp!]
        \centering
        \begin{subfigure}[b]{0.48\textwidth}
            \centering
            \includegraphics[width=\textwidth]{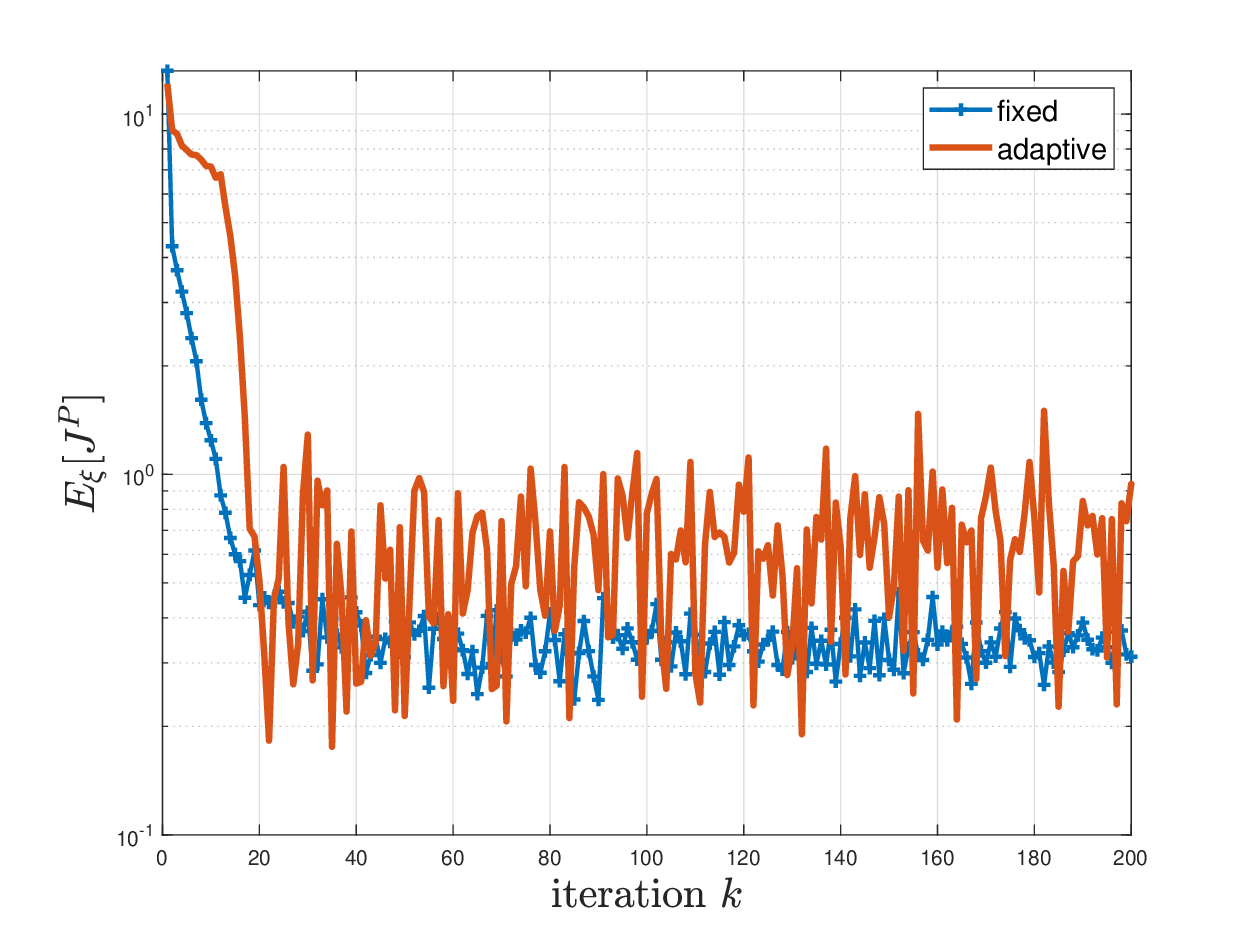}
            \caption{Mean compliance.}
        \end{subfigure}
        \begin{subfigure}[b]{0.48\textwidth}
            \centering
            \includegraphics[width=\textwidth]{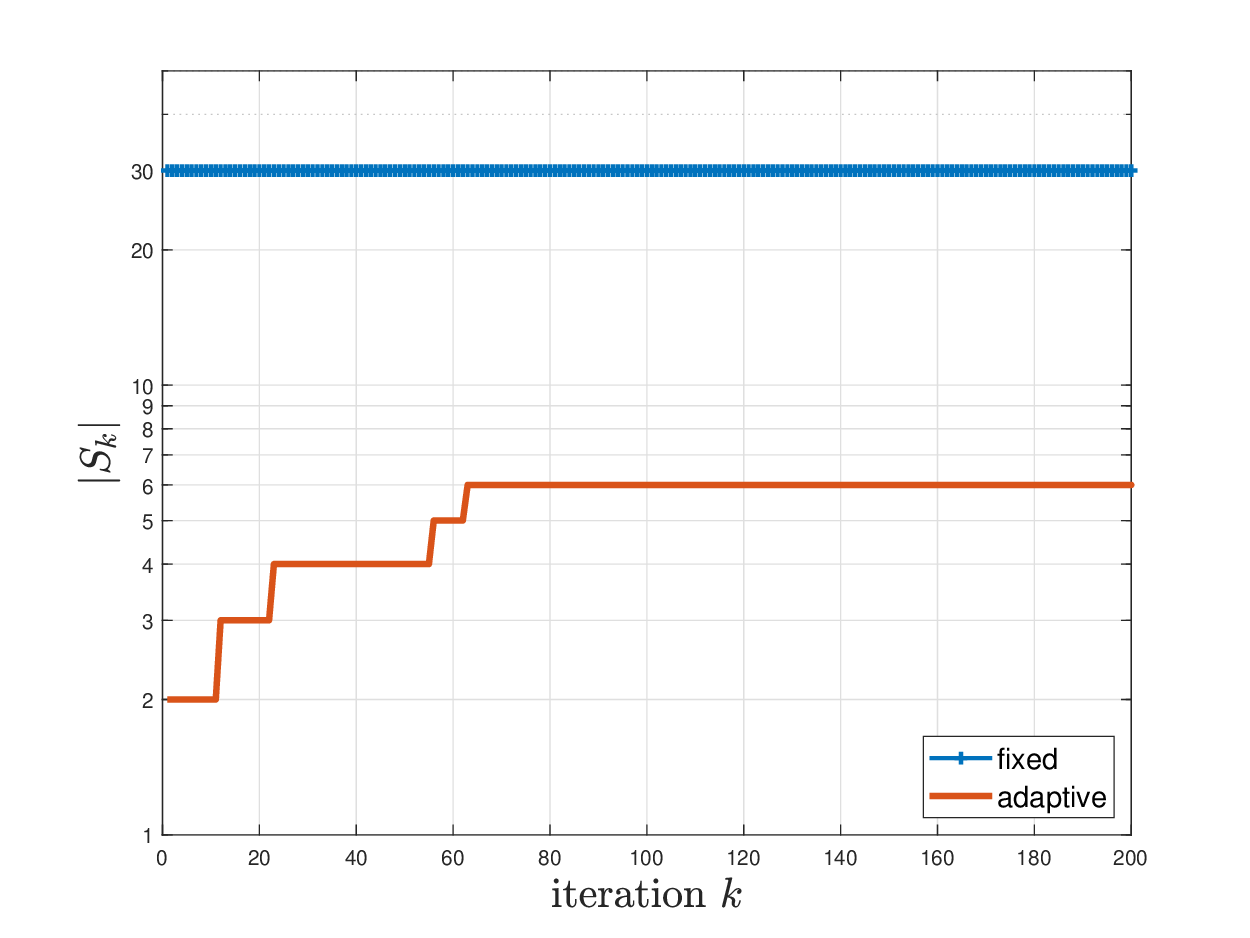}
            \caption{Sample size.}
        \end{subfigure}       
        \begin{subfigure}[b]{0.48\textwidth}
            \centering
            \includegraphics[width=\textwidth]{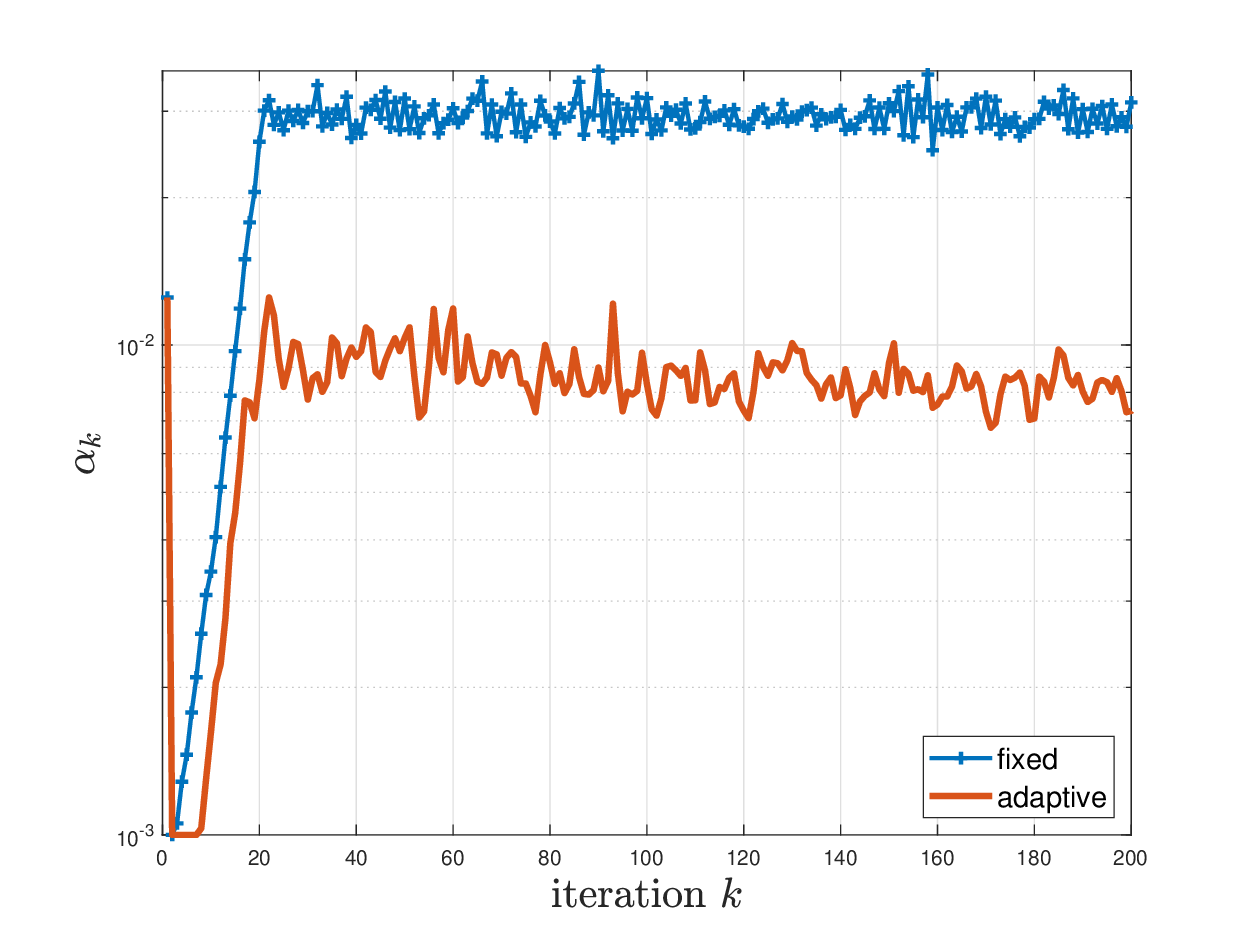}
            \caption{Step length.}
        \end{subfigure}
            \caption{Evolution of the mean compliance $\mathbb{E}_{\xi}[J^P]$, the Monte Carlo sample size $|S_k|$, and the step length $\alpha_k$ for the fixed and adaptive approaches.}
            \label{fig:combined_evolution}
\end{figure}

\begin{table}[htp!]
	\caption{Comparison of the final values of the optimization quantities in the fully adaptive study.}\label{table:combined_final_values}
	\centering
\begin{tabular}{lcc}
	Quantity & Adaptive Mesh \& Sampling & Fine Mesh \& Full Sampling  \\ \hline
    $\E_{\xi}[J^P]$ & 0.94 & 0.37 \\
    $\cQ$ in \eqref{eqn:Q_bound} & 0.07 & 0.002 \\
    $\eta^c$  & 0.09 & 0.02 \\
    $\eta^d$ & 0.30 & 0.04 \\
    $|S|$ & $2 \rightarrow 6$ & 30 \\
    $\alpha_k$ & 0.007 & 0.031 \\
    $\E_{\xi}[\diff J^P]$& 16.35 & 8.80 \\
    $\mathbb{V}_{\xi}[\diff J^P]$& 264.6 & 486.9
    \end{tabular}
\end{table}

\begin{table}[htp!]
	\caption{Comparison of the computational effort between the fixed fine mesh under full sampling and the  adaptive mesh and sampling strategy in terms of the degrees of freedom, number of samples, computational runtime, and the computational index defined in~\eqref{index}.}\label{table:combined_comp_index}
	\centering
\begin{tabular}{l|l|l|l|l}
	                            & DoF                              & $|S_k|$  & $T_{\text{total}}$[hours]  & \texttt{CI} \\ \hline
    fixed     & $260,\!282$                      & $30$              & $45.6$              & $119 \cdot 10^{10}$ \\
    adaptive   & $29,\!162 \rightarrow 43,\!378$  & $2 \rightarrow 6$ & $8.72$              & $2.42 \cdot 10^{10}$
    \end{tabular}
\end{table}

 \begin{figure}[htp!]
         \centering
         \begin{subfigure}[b]{0.25\textwidth}
             \centering
             \includegraphics[width=\textwidth]{fs_umf_design_rev.png}
             \caption{Fixed design.}
         \end{subfigure}
         \qquad        
         \begin{subfigure}[b]{0.25\textwidth}
             \centering
             \includegraphics[width=\textwidth]{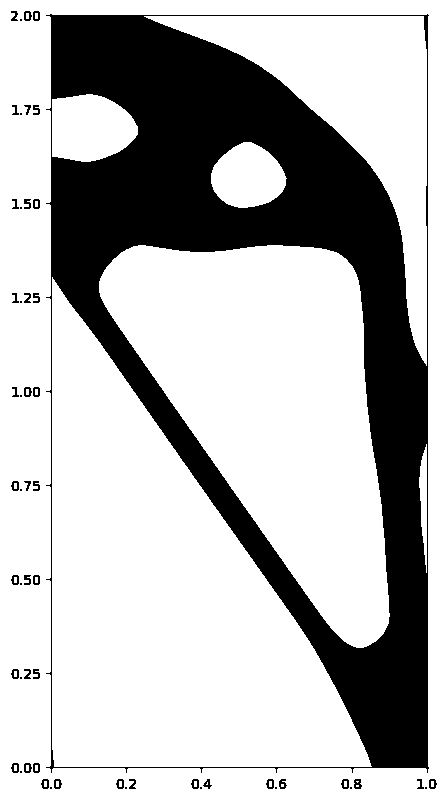}
             \caption{Adaptive design.}
         \end{subfigure}
          \qquad  
         \begin{subfigure}[b]{0.25\textwidth}
             \centering
             \includegraphics[width=\textwidth]{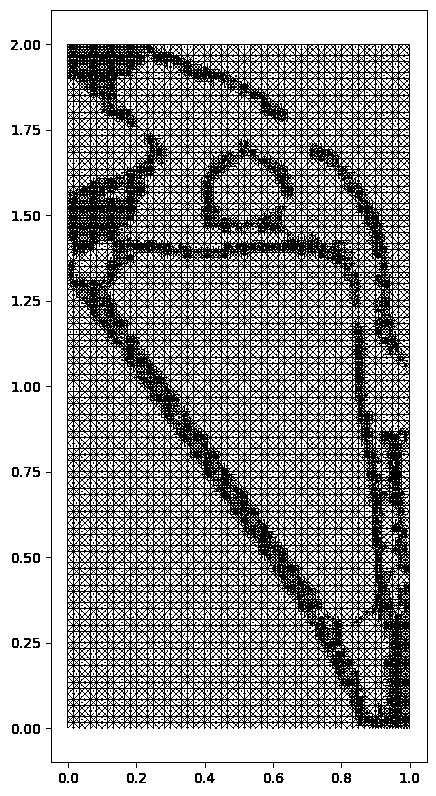}
             \caption{Adaptive mesh.}
         \end{subfigure}
         \caption{Final optimized designs obtained on the fixed fine mesh with $260{,}282$ degrees of freedom using full Monte Carlo sampling (a) and on an adaptively generated mesh with $43{,}378$ degrees of freedom using adaptive sampling and adaptive mesh refinement (b). The corresponding adaptive mesh is shown in (c).}
         \label{fig:combined_final_design}
\end{figure}


\section{Conclusions}\label{sec:conclusion}

In this paper, we have presented an adaptive framework for robust shape optimization governed by a linear elasticity model with random inputs. We have established the existence of the shape derivative and derived its explicit representation for the expected compliance minimization problem subject to a penalized volume constraint. Within this adaptive framework, the number of Monte Carlo samples, the mesh resolution, and the step length in the gradient-based optimization algorithm are all controlled by a posteriori error estimators. In addition, we have employed a multi-goal-oriented estimator derived from the discretization of both the state equation and the deformation problem. To evaluate the performance of the proposed method, we have optimized the shape of leg-like structural components to minimize compliance during touchdown under uncertain contact forces. Numerical experiments demonstrate that the adaptive procedure significantly reduces computational cost while accurately tracking the error behavior, yielding results comparable to those obtained using a fixed fine mesh with full sampling. Future work may focus on reducing the dependence of the final designs on the initial level-set configuration, a behavior observed in the numerical experiments. Possible remedies include adopting a reaction--diffusion-type level-set formulation or enhancing the stability of the Hamilton--Jacobi equation by introducing an artificial diffusion term. Moreover, the adaptive framework could be improved by introducing a dedicated goal functional to better control discretization errors associated with the Hamilton--Jacobi equation. \x{Another promising direction for future research is the extension of the proposed framework to three-dimensional shape optimization problems, where the benefits of adaptive sampling and adaptive mesh refinement are expected to become even more significant due to the substantially increased computational complexity.}

\begin{acknowledgements}
This work has been supported by the Scientific Research Projects Coordination Unit of Middle East Technical University, Ankara, Türkiye under Grant No: Genel-705-2025-11694. HY also gratefully acknowledges the support provided by the Scientific and Technological Research Council of Türkiye (TÜBİTAK) through the  2219–International Postdoctoral Research Fellowship Program (Project No: 1059B192302207), and would like to thank the Max Planck Institute for Dynamics of Complex Technical Systems, Magdeburg for its excellent hospitality. 
\end{acknowledgements}


\appendix  
\section*{Appendix: Proof of Proposition~\ref{prop:derivative_penalizedobjective}} \label{appendix:derivative_penalizedobjective}

Let $\cW_t = \Phi_t(\cW) \subset \cD$ be a parameterized domain generated by the flow $\Phi_t$ associated with a deformation field $\theta \in \Theta^k(\cD)$. Since $\theta = 0$ on $\Gamma_d \cup \Gamma_n$, it follows that $\Phi_t = \mathrm{Id}$ on $\Gamma_d \cup \Gamma_n$. Further, let $u_t(\omega) \in H^1_d(\cD)^2$  be the weak solution of \eqref{eqn:elastic_PDE_ersatz} with $\cW$ replaced by $\cW_t$. 
To facilitate  differentiation with respect to $t$, we perform a change of variables $x \mapsto \Phi_t(x)$, thereby obtaining an expression that depends on  $A_\cW$ and $f_\cW$ instead of $A_{\cW_t}$ and $f_{\cW_t}$. Applying the chain rule to  \x{the pullback variable $\widetilde{u}_t := u_t \circ \Phi_t$} yields
\begin{equation*}
        \nabla \widetilde{u}_t = \nabla (u_t \circ \Phi_t) = (\nabla u_t) \circ \Phi_t \, \nabla \Phi_t,
\end{equation*}
and we define
\begin{equation} \label{eqn:transformation}
    \begin{split}
        \varepsilon(t, \widetilde{u}_t) &:= \nabla^s u_t \circ \Phi_t 
        = \frac{1}{2} \Big(\nabla \widetilde{u}_t \, \nabla \Phi_t^{-1} + (\nabla \widetilde{u}_t \, \nabla \Phi_t^{-1})^T \Big).
    \end{split}
\end{equation}
We now follow the formal Lagrangian approach and  define the Lagrangian $\mathcal{L}(\cdot;\omega) : [0, \tau] \times H^1_d(\cD)^2 \times H^1_d(\cD)^2 \ra \R$
in view of \eqref{eqn:transformation},  
\begin{equation}
        \label{eqn:sd1}
        \begin{split}
            \mathcal{L}(t, v, q; \omega) & = \int_\cD f_\cW(\omega) \cdot v \, \zeta(t) \, dx  + \int_{\Gamma_n} g(\omega) \cdot v \, ds \\
   &\quad + \int_\cD \big( A_\cW(\omega) \varepsilon(t, v) : \varepsilon(t, q) - f_\cW(\omega) \cdot q  \big) \, \zeta(t) \, dx \\
    &\quad        - \int_{\Gamma_n} g(\omega) \cdot q \, ds,
        \end{split}
\end{equation}
where $\zeta(t)$ denotes the Jacobian determinant of the transformation $x \mapsto \Phi_t(x)$  and $\varepsilon(0, v) = \nabla^s v$. Note that the Jacobian determinant $\zeta(t)$ and the mapping $\Phi_t$ do not appear in the boundary integrals since $\Phi_t = \mathrm{Id}$ on $\Gamma_n$.

One can easily see that  $J(\cW_t, \omega) = \mathcal{L}(t, \widetilde{u}_t, q; \omega)$ for all $q \in H^1_d(\cD)^2$. Using the averaged-adjoint framework, one obtains
\begin{equation}
        \label{eqn:sd2}
        \begin{split}
            \diff J(\cW,\omega)(\theta) = \frac{\diff}{\diff t} \mathcal{L}(t, \widetilde{u}_t, z;\omega)\rvert_{t=0} 
            = \partial_t \mathcal{L}(0, \widetilde{u}_0, z;\omega),
        \end{split}
\end{equation}
which allows us to avoid the explicit differentiation of the state variable $\widetilde{u}_t$. In \eqref{eqn:sd2}, by noting that $u=\widetilde u_0$, the adjoint variable $z$ is determined from the first-order optimality condition $\partial_v \mathcal{L}(0, u, z;\omega)(\varphi) = 0 \;\,  \forall  \varphi \in H^1_d(\cD)^2$.  By using the symmetry property $A_\cW = A_\cW^T$ for each $\omega \in \Omega$ and noting that $\zeta(0)=1$, \eqref{eqn:sd1} yields
\begin{equation}
    \label{eqn:sd3}
    \begin{split}
        \partial_v \mathcal{L}(0, u, z;\omega)(\varphi)  & = \int_\cD f_\cW(\omega) \cdot \varphi \, dx + \int_{\Gamma_n} g(\omega) \cdot \varphi \, ds  
         \\
        & \quad + \int_\cD \big( A_\cW(\omega) \varepsilon(0, \varphi) : \varepsilon(0, z) \big) \, dx
     \end{split}
\end{equation}
for all $\varphi \in H^1_d(\cD)^2$. Assuming sufficient regularity, integration by parts yields the following strong form of the adjoint problem:
\begin{subequations}\label{eqn:adjoint_strong}
\begin{eqnarray}
        -\Div(A_\cW(x,\omega) \nabla^s z) &=& -f_\cW(x,\omega) \qquad \, \textnormal{in} \ \cD, \\
        z &=& 0 \qquad \qquad \qquad \; \textnormal{on} \ \Gamma_d, \\
        (A_\cW(x,\omega) \nabla^s z) n &=& -g(x,\omega) \qquad \;\;\;\,  \textnormal{on} \ \Gamma_n, \\
     (A_{\cW}(x,\omega) \nabla^s z) \, n &=& 0 \qquad  \qquad \quad \;\;\;\;\, \textnormal{on} \;\;  \partial \cD \backslash ( \overline \Gamma_d \cup \overline \Gamma_n).
\end{eqnarray}
\end{subequations}
Next, the shape derivative from \eqref{eqn:sd1}  can be computed as follows:
\begin{align}  
    \label{eqn:sd4}
    &\diff J(\cW,\omega)(\theta)  = \partial_t \mathcal{L}(0, u, z;\omega) \nonumber \\
    &\quad = \int_\cD \big( f_\cW(\omega) \cdot (u-z)\, \zeta'(0)  \big)\, dx  \nonumber \\
    &\qquad + \int_\cD \Big( 
        \big( A_\cW(\omega)\, \partial_t \varepsilon(0, u) : \varepsilon(0, z)
        + A_\cW(\omega)\, \varepsilon(0, u) : \partial_t \varepsilon(0, z) \big)\, \zeta(0)  \nonumber \\
    &\qquad\qquad + \big( A_\cW(\omega)\, \varepsilon(0, u) : \varepsilon(0, z) \big)\, \zeta'(0)
        \Big)\, dx.
\end{align}
Recalling that $\zeta(0) = 1$ and $\zeta'(0) = \Div \theta$, an application of the chain rule yields
\begin{equation}\label{eqn:sd5}
    \begin{split}
        \partial_t \varepsilon(0, u) 
        &= -\frac{1}{2}\Big( \nabla u\, \nabla \theta 
        + (\nabla u\, \nabla \theta)^T \Big), \\
            \partial_t \varepsilon(0, z) 
        &= -\frac{1}{2}\Big( \nabla z\, \nabla \theta 
        + (\nabla z\, \nabla \theta)^T \Big).
    \end{split}
\end{equation}
Then, substituting \eqref{eqn:sd5} into \eqref{eqn:sd4}, we obtain
\begin{align*}
        \diff J(\cW,\omega)(\theta)
        &= \int_\cD f_\cW(\omega)\cdot (u-z)\,\Div\theta \, dx \\
        &\quad
        - \int_\cD
        \Big(
        \nabla u^T \big(A_\cW(\omega)\nabla^s z\big)
        + \nabla z^T \big(A_\cW(\omega)\nabla^s u\big)
        \Big)
        : \nabla\theta \, dx \\
        &\quad
        + \int_\cD
        A_\cW(\omega)\nabla^s u : \nabla^s z \,
        \Div\theta \, dx .
\end{align*}
Using the symmetry property $(A_\cW \nabla^s u)^T = A_\cW \nabla^s u$ 
together with the identity $z = -u$ obtained from 
\eqref{eqn:elastic_PDE_ersatz} and \eqref{eqn:adjoint_strong}, 
we arrive at
\begin{align}\label{eqn:sd7}
        \diff J(\cW,\omega)(\theta)
        &= \int_\cD \big( 2 \nabla u^T A_\cW(\omega) \nabla^s u 
        + \left( 2 f_\cW(\omega) \cdot u 
        - A_\cW(\omega) \nabla^s u : \nabla^s u \right) \mathrm{Id} \big) 
        : \nabla \theta \, dx
\end{align}
for all $\theta \in \Theta^k(\cD)$.

For the shape derivative of the second term in \eqref{eqn:optimization_penalized_fixed}, we use the distributed expression for the shape derivative of the volume \cite[Theorem~4.1]{MCDelfour_JPZolesio_2011}, namely,
\begin{equation}\label{eqn:sd8}
 \diff \mathcal{V}(\cW)(\theta) = \int_\cW \Div \theta \, dx, \quad \hbox{where} \quad \mathcal{V}(\cW) = \int_\cW dx.
\end{equation}
Finally,  combining \eqref{eqn:sd7} and \eqref{eqn:sd8}, the desired result follows.




\end{document}